\documentclass[11pt,twoside,a4paper]{amsart}
\usepackage{amssymb}
\usepackage{amsmath,amscd}
\usepackage[initials,nobysame,shortalphabetic]{amsrefs}
\usepackage[all,cmtip]{xy}

\usepackage[utf8]{inputenc}

\usepackage{geometry}
\def\End{{\rm End}}

\def\deg{\text{deg}\,}

\def\dbar{\bar\partial}
\def\R{{\mathbb R}}
\def\C{{\mathbb C}}

\def\Cn{\C^n}

\def\D{{\mathcal D}}

\def\PM{{\mathcal{PM}}}

\def\Hom{{\rm Hom\, }}
\def\codim{{\rm codim\,}}

\def\Im{{\rm Im\, }}

\def\Z{{\mathbb Z}}

\def\Ok{{\mathcal O}}

\newcommand{\Com}[1]{}
\DeclareMathOperator{\Id}{Id}
\DeclareMathOperator{\supp}{supp}
\DeclareMathOperator{\ann}{ann}
\DeclareMathOperator{\rank}{rank}
\DeclareMathOperator{\im}{im}
\DeclareMathOperator{\coker}{coker}

\DeclareMathOperator{\Ext}{Ext}

\def\be{\begin{equation}}
\def\ee{\end{equation}}

\newtheorem{thm}{Theorem}[section]
\newtheorem{lma}[thm]{Lemma}
\newtheorem{cor}[thm]{Corollary}
\newtheorem{prop}[thm]{Proposition}

\theoremstyle{definition}

\theoremstyle{remark}

\newtheorem{preremark}[thm]{Remark}
\newtheorem{preex}[thm]{Example}

\newenvironment{remark}{\begin{preremark}}{\end{preremark}}
\newenvironment{ex}{\begin{preex}}{\end{preex}}

\numberwithin{equation}{section}

\begin{document}

\title{A comparison formula for residue currents}

\date{\today}

\author{Richard L\"ark\"ang}

\thanks{The author was supported by the Swedish Research Council.}

\address{Richard L\"ark\"ang\\ Department of
  Mathematics\\Chalmers University of Technology and the University of
  Gothenburg\\412 96 G\"oteborg\\Sweden}

\email{larkang@chalmers.se}


\begin{abstract}
    Given two ideals $\mathcal{I}$ and $\mathcal{J}$ of holomorphic functions such that
    $\mathcal{I} \subseteq \mathcal{J}$, we describe a comparison formula relating the
    Andersson-Wulcan currents of $\mathcal{I}$ and $\mathcal{J}$.
    More generally, this comparison formula holds for residue currents associated
    to two generically exact Hermitian complexes together with a morphism
    between the complexes.

    One application of the comparison formula is a generalization
    of the transformation law for Coleff-Herrera products to Andersson-Wulcan currents
    of Cohen-Macaulay ideals. We also use it to give an analytic proof by means of residue currents
    of theorems of Hickel, Vasconcelos and Wiebe related to the Jacobian ideal of a holomorphic mapping.
\end{abstract}

\maketitle

\section{Introduction}

The theory of residue currents of Coleff-Herrera, Dickenstein-Sessa,
Passare-Tsikh-Yger, Andersson-Wulcan and others
has provided a strong tool for proving different results.
For example, it has been used to prove
results about membership problems in commutative algebra, including Brian\c{c}on-Skoda type results in \cites{ASS,AW3,Sz}.
However, there are similar results which appear natural to approach by such methods,
but which have so far not been possible to prove in this way due to lack of precise enough description of the involved 
residue currents.

In this paper we introduce a comparison formula for residue currents, generalizing the
classical transformation law for complete intersections, which allows for expressing
residue currents in \cite{AW1} and \cite{PTY} in terms of ``simpler'' currents.
In Section~\ref{ssect:intro-trans} to Section~\ref{ssect:intro-prescribed} we
discuss various applications of this formula. Some of the applications are elaborated
in this article, others are from later work after the appearance of the first version
of this article.
One application is that the comparison formula gives precise enough information about 
residue currents to give analytic proofs of theorems of Hickel, Vasconcelos and Wiebe,
Theorem~\ref{thm:vasc} and Corollary~\ref{corhickel}. These results 
had previously only been proven by algebraic means.
Other applications of the comparison formula include the results in \cite{Lar4}, where it
is used to construct residue currents with prescribed annihilator
ideals on singular varieties, and in \cite{LW}, where it is used to obtain
precise descriptions of residue currents associated to Artinian monomial ideals.

\subsection{The transformation law}

We begin by recalling the transformation law, which our formula is a generalization of.
Let $f = (f_1,\dots,f_p)$ be a tuple of germs of holomorphic functions at the origin in $\Cn$
defining a complete intersection, i.e., so that $\codim Z(f) = p$. Associated to $f$, there exists
a current
\begin{equation} \label{comp:eqchp}
    \mu^f = \dbar \frac{1}{f_p}\wedge \dots\wedge \dbar \frac{1}{f_1},
\end{equation}
called the \emph{Coleff-Herrera product} of $f$, which was introduced in \cite{CH}.
We let $\ann_{\Ok} \mu^f$ be the annihilator of $\mu^f$, i.e.,
the holomorphic functions $g$ such that $g \mu^f = 0$, and we let $\mathcal{J}(f)$
be the ideal generated by $f$.
One of the fundamental properties of the Coleff-Herrera product is the \emph{duality theorem},
which says that $\ann_{\Ok} \mu^f = \mathcal{J}(f)$. The duality theorem was proven 
independently by Dickenstein and Sessa, \cite{DS}, and Passare, \cite{PMScand}.

Another fundamental property of the Coleff-Herrera product is that it satisfies
the \emph{transformation law}. Earlier versions of the transformation law
involving cohomological residues (Grothendieck residues) exist,
see for example \cite{Tong}*{(4.3)} and \cite{GH}*{p.~657}.
\begin{thm}
    Let $f = (f_1,\dots,f_p)$ and $g = (g_1,\dots,g_p)$ be tuples of holomorphic functions defining
    complete intersections. Assume there exists a matrix $A$ of holomorphic functions such that
    $f = g A$. Then
    \begin{equation*}
        \dbar \frac{1}{g_p}\wedge\dots\wedge\dbar\frac{1}{g_1}=
        (\det A) \dbar \frac{1}{f_p}\wedge\dots\wedge\dbar\frac{1}{f_1}.
    \end{equation*}
\end{thm}
In the setting of Coleff-Herrera products the transformation law was first stated in
\cite{DS}, and it was explained that the proof can be reduced to the absolute case (when $p = n$)
and cohomological residues together with the technique from \cite{CH}
of fibered residues. An elaboration of this proof can be found in \cite{DS2}.

For cohomological residues as in \cite{GH} the idea of the proof is that if
$dg_1\wedge\dots\wedge dg_n$ is non-vanishing and $A$ is invertible, then
the transformation law is essentially the change of variables formula for
integrals.

In the case when $p = n$ the transformation law combined with the Nullstellensatz
allow to express in an explicit fashion the action of $\mu^f$, see for example
\cite{TsikhBook}*{p.~22}. Essentially the same idea is also used in \cite{GH}
to prove the duality theorem for Grothendieck residues by using the transformation law.

One particular case of the transformation law is when we choose different generators
$f' = (f'_1,\dots,f'_p)$ of the ideal generated by $f$. Then the Coleff-Herrera product
of $f'$ differs from the one of $f$ only by an invertible holomorphic function, and
hence it can essentially be considered as a current associated to the ideal
$\mathcal{J}(f)$.

The requirement that $f = gA$ means that $\mathcal{J}(f)\subseteq\mathcal{J}(g)$.
If we consider the Coleff-Herrera product of $g$ as a current associated
to the ideal $\mathcal{J}(g)$, then the transformation law says that the inclusion
$\mathcal{J}(f) \subseteq \mathcal{J}(g)$ implies that we can
express the Coleff-Herrera product of $\mathcal{J}(g)$ in terms of
the Coleff-Herrera product of $\mathcal{J}(f)$.

\subsection{A comparison formula for Andersson-Wulcan currents}

Consider an arbitrary ideal $\mathcal{J} \subseteq \Ok = \Ok_{\Cn,0}$ of holomorphic functions.
Throughout this article we let $\Ok$ denote $\Ok_{\Cn,0}$, the ring of germs of holomorphic
functions at the origin in $\Cn$, unless otherwise stated.
Let $(E,\varphi)$ be a \emph{Hermitian resolution} of $\Ok/\mathcal{J}$,
\begin{equation*}
    0 \to E_N \xrightarrow[]{\varphi_N} E_{N-1} \to \dots \xrightarrow[]{\varphi_1} E_0 \to \Ok/\mathcal{J} \to 0,
\end{equation*}
i.e., a free resolution of $\Ok/\mathcal{J}$ where the free modules are equipped
with Hermitian metrics.
Given $(E,\varphi)$, Andersson and Wulcan constructed in \cite{AW1} a current $R^E$
such that $\ann_{\Ok} R^E = \mathcal{J}$, where $R^E = \sum_{k=p}^N R^E_k$,
$p = \codim Z(\mathcal{J})$, and $R^E_k$ are $\Hom(E_0,E_k)$-valued $(0,k)$-currents.
We will sometimes denote the current $R^E$ by $R^{\mathcal{J}}$, although it
depends on the choice of Hermitian resolution $E$ of $\Ok/\mathcal{J}$.
We refer to Section~\ref{sectawcurrents} for a more thorough description of the
current $R^E$.
As mentioned above, such currents have been used to study membership problems.
Another important application has been to construct solutions to the 
$\dbar$-equation on singular varieties, \cites{AS2,AS3}.

In case $\mathcal{J}$ is a complete intersection defined by a tuple $f$,
then $\mathcal{J}$ has an explicit free resolution; the Koszul complex of $f$.
In that case, the Andersson-Wulcan current associated to the Koszul complex coincides
with the Coleff-Herrera product of $f$, see Section~\ref{ssectbm}.

We now consider two ideals $\mathcal{I}$ and $\mathcal{J}$ such that $\mathcal{I}\subseteq \mathcal{J}$,
and free resolutions $(E,\varphi)$ and $(F,\psi)$ of $\Ok/\mathcal{J}$ and $\Ok/\mathcal{I}$ respectively.
If we choose minimal free resolutions, then in particular $\rank E_0 = \rank F_0 = 1$, i.e., $E_0 \cong \Ok \cong F_0$,
and we let $a_0 : F_0 \to E_0$ be this isomorphism.
Since $\mathcal{I}\subseteq\mathcal{J}$, we have the natural surjection $\pi : \Ok/\mathcal{I} \to \Ok/\mathcal{J}$,
and by the choice of $a_0$, the diagram
\begin{equation} \label{eqa0morphism}
\begin{gathered}
\xymatrix{
E_0 \ar[r] &\Ok/\mathcal{J} \\
F_0 \ar[u]^{a_0}  \ar[r] & \Ok/\mathcal{I} \ar[u]^{\pi}
}
\end{gathered}
\end{equation}
commutes. In fact, even when $(E,\varphi)$ and $(F,\psi)$ are not minimal, one can always
find $a_0$ making \eqref{eqa0morphism} commute, and we thus assume $a_0$ is chosen in this way.
Using the fact that the $F_k$ are free and that $(E,\varphi)$
is exact, by a simple diagram chase one can complete this
to a commutative diagram
\begin{equation} \label{eqamorphism}
\begin{gathered}
\xymatrix{
0 \ar[r] & E_N \ar[r]^{\varphi_N} & E_{N-1} \ar[r]^{\varphi_{N-1}}& \cdots \ar[r]^{\varphi_1} & E_0 \ar[r] &\Ok/\mathcal{J} \ar[r] & 0 \\
0 \ar[r] & F_N \ar[r]^{\psi_N} \ar[u]^{a_N}& F_{N-1}\ar[r]^{\psi_{N-1}} \ar[u]^{a_{N-1}} & \cdots \ar[r]^{\psi_1}& F_0 \ar[u]^{a_0}  \ar[r] & \Ok/\mathcal{I} \ar[u]^{\pi}\ar[r] & 0.
}
\end{gathered}
\end{equation}
The commutativity means that $a : (F,\psi) \to (E,\varphi)$ is a morphism of complexes,
cf., Proposition~\ref{propcomplexcomparison}.

The main result of this article is a comparison formula for the currents associated to $\mathcal{I}$
and $\mathcal{J}$ obtained from the morphism $a$. The formula involves forms $u^E$ and $u^F$,
which are certain endomorphism-valued forms on the free resolutions $E$ and $F$.
These forms are smooth outside of $Z(\mathcal{I}) \cup Z(\mathcal{J})$;
see Section~\ref{sectawcurrents} for details about how they are defined.
Throughout the article, $\chi(t) : \R_{\geq 0} \to \R_{\geq 0}$
is a smooth cut-off function such that $\chi(t) \equiv 0$ for $t \ll 1$ and $\chi(t) \equiv 1$
for $t \gg 1$.

\begin{thm} \label{thmRcomparisonIdeals}
    Let $\mathcal{I},\mathcal{J} \subseteq \Ok$ be two ideals such that $\mathcal{I} \subseteq \mathcal{J}$,
    and let $(E,\varphi)$ and $(F,\psi)$ be Hermitian resolutions of $\Ok/\mathcal{J}$ and $\Ok/\mathcal{I}$ respectively.
    Let $a : (F,\psi) \to (E,\varphi)$ be the morphism in \eqref{eqamorphism} induced by
    the natural surjection $\pi : \Ok/\mathcal{I} \to \Ok/\mathcal{J}$. Then,
    \begin{equation} \label{eqRcomparison}
        R^\mathcal{J} a_0 - a R^\mathcal{I} = \nabla_{\varphi} M,
    \end{equation}
    where $\nabla_{\varphi} = \sum \varphi_k - \dbar$, and
    \begin{equation*}
        M = \lim_{\epsilon\to 0^+} \dbar\chi(|h|^2/\epsilon) \wedge u^E a u^F,
    \end{equation*}
    where $h$ is a tuple of holomorphic functions such that $h \not\equiv 0$,
    and $\{ h = 0 \}$ contains $Z(\mathcal{I}) \cup Z(\mathcal{J})$.
\end{thm}

The theorem in fact holds in a more general setting. First of all, there are
Andersson-Wulcan currents associated not just to Hermitian resolutions, but to any generically
exact Hermitian complex. The theorem holds for such residue currents
together with arbitrary morphisms of the complexes, Theorem~\ref{thmRcomparison}.
In addition, the current $M$ is there interpreted as the so-called residue of
an almost semi-meromorphic current.
To elaborate more precisely how $M$ and $\nabla_\varphi$ are defined,
more background from the construction of the Andersson-Wulcan currents is required.
We refer to Section~\ref{sectawcurrents} for the necessary background,
and Section~\ref{sectcomparisonformula} for a more precise statement of the
comparison formula in the general form.

\subsection{A transformation law for Andersson-Wulcan currents associated with Cohen-Macaulay ideals} \label{ssect:intro-trans}

Our first application is a situation in which the current $M$ in \eqref{eqRcomparison} vanishes.
This gives a direct generalization of the transformation law for Coleff-Herrera
products to Andersson-Wulcan currents associated with Cohen-Macaulay ideals.
We recall that an ideal $\mathcal{J}$ is \emph{Cohen-Macaulay} if $\Ok/\mathcal{J}$ has a
free resolution of length equal to $\codim Z(\mathcal{J})$.

\begin{thm} \label{thmtranscm}
    Let $\mathcal{I},\mathcal{J} \subseteq \Ok$ be two Cohen-Macaulay ideals
    of the same codimension $p$ such that $\mathcal{I} \subseteq \mathcal{J}$.
    Let $(F,\psi)$ and $(E,\varphi)$ be
    Hermitian resolutions of length $p$ of $\Ok/\mathcal{I}$ and $\Ok/\mathcal{J}$ respectively.
    If $a : (F,\psi) \to (E,\varphi)$ is the morphism in \eqref{eqamorphism}
    induced by the natural surjection $\pi : \Ok/\mathcal{I} \to \Ok/\mathcal{J}$, then
    \begin{equation*}
        R_p^\mathcal{J} a_0 = a_p R_p^\mathcal{I}.
    \end{equation*}
\end{thm}

The proof of Theorem~\ref{thmtranscm} is given in Section~\ref{secttranscm};
it is a special case of the more general Theorem~\ref{thmtranscmgen}.
In Remark~\ref{remtransl} in Section~\ref{secttranscm}, we describe how
the transformation law for Coleff-Herrera products is a special case of
Theorem~\ref{thmtranscm}.

In the article \cite{DS2} two proofs of the transformation law for Coleff-Herrera
products are given. One of the proofs can in fact be adapted to give an alternative
proof of Theorem~\ref{thmtranscm}, see Section~\ref{secttranscm}.

See Section~\ref{secttranscm} for various examples of how one can use
Theorem~\ref{thmtranscm} or its generalization Theorem~\ref{thmtranscmgen}
to express the current $R^\mathcal{I}$ for a Cohen-Macaulay ideal $\mathcal{I}$
in terms of other currents in an explicit way.
This type of expressions were used by Lejeune-Jalabert in \cite{LJ2} to create certain 
cohomological residues for Cohen-Macaulay ideals in terms of Grothendieck residues.
She used this type of residues to express the fundamental cycle of such ideals in terms
of Grothendieck residues. However, duality properties of such cohomological residues were not investigated.
Lundqvist, \cites{Lund1,Lund2}, also constructed cohomological residues associated to pure dimensional ideals,
and proved that they satisfy a duality theorem.
With the help of the comparison formula, we elaborate in \cite{LarG} a bit on the relation between such residues,
and the relation with Andersson-Wulcan currents.
The comparison formula also plays an important role in that article, as it is used to prove
functoriality for a pairing defined with the help of Andersson-Wulcan currents.

In Section~\ref{sectnoncmex} we give an example of a computation when the ideal
is not Cohen-Macaulay.

In the joint article \cite{LW2} with Wulcan we use Theorem~\ref{thmtranscm} to
express explicitly the fundamental cycle of a pure dimensional ideal in terms
of residue currents, generalizing the Poincar\'e-Lelong formula. This is related
to the construction of Lejeune-Jalabert mentioned above. In another joint article,
\cite{LW}, we use Theorem~\ref{thmtranscm} to calculate in a simpler and in some aspects more explicit
way residue currents associated to Artinian monomial ideals, compared to earlier
work by Wulcan. Having such explicit expression for the currents, we were able to directly
prove the results from \cite{LW2} for such ideals.

\subsection{The Jacobian determinant of a holomorphic mapping}

Let \linebreak $f = (f_1,\dots,f_m) \in \Ok^{\oplus m}$.
Let $\mathcal{J}ac(f)$ be the ideal generated by the coefficients of $df_1 \wedge \dots \wedge df_m$, i.e., if
\begin{equation*}
    df_1\wedge\dots\wedge df_m = \sum_{|I|=m} f_I dz_{i_1} \wedge \dots \wedge dz_{i_m},
\end{equation*}
then $\mathcal{J}ac(f)$ is the ideal generated by all the $f_I$'s.

We give an analytic proof of the following (slightly weaker variant of a) theorem of Vasconcelos, \cite{Vasc}*{Theorem~(2.4)},
using the generalization Theorem~\ref{thmRcomparison} of Theorem~\ref{thmRcomparisonIdeals}.
In \cite{Vasc} this theorem was proved for the polynomial ring over a field.
In \cite{Wi} Wiebe proved this theorem (formulated slightly differently) in the case $m = n$
for any local ring.
We recall that if $I$ and $J$ are ideals in a ring $R$, then the ideal quotient $I : J$ is the ideal
\begin{equation*}
    I : J := \{ r \in R \mid rJ \subseteq I \}.
\end{equation*}

\begin{thm} \label{thm:vasc}
    Let $f = (f_1,\dots,f_m)$ be a tuple of holomorphic functions in $\Ok$ vanishing at $\{ 0 \}$,
    and assume that $\Ok/\mathcal{J}(f)$ has a free resolution of length $\leq m$.
    Let $\mathcal{J}_m(f)$ be the ideal of all holomorphic functions vanishing at all irreducible
    irreducible components of $Z(f)$ of codimension $m$.
    Then,
    \begin{equation*}
        \mathcal{J}_m(f) = \mathcal{J}(f) : \mathcal{J}ac(f).
    \end{equation*}
\end{thm}

Note that if $\mathcal{I}$ and $\mathcal{J}$ are ideals in $\Ok$,
then $\mathcal{J} : \mathcal{I} = \Ok$ if and only if $\mathcal{I} \subseteq \mathcal{J}$,
and that $\mathcal{J}_m(f) = \Ok$ if and only if $Z(f)$ has no irreducible components of codimension $m$.
Combining these two remarks with the theorem, one gets that \emph{$\mathcal{J}ac(f) \subseteq \mathcal{J}(f)$ if and only if 
$Z(f)$ has no irreducible component of codimension $m$} (under the assumption that $\Ok/\mathcal{J}(f)$
has a free resolution of length $\leq m$).

Note that if $f = (f_1,\dots,f_n)$, then $\mathcal{J}ac(f)$ is generated by the Jacobian determinant $J_f$ of $f$.
Moreover, by the Hilbert syzygy theorem, $\Ok/\mathcal{J}(f_1,\dots,f_m)$ always has a free resolution of length $n$.
Finally, if $f=(f_1,\dots,f_n)$ vanishes at $0$, then $\codim Z(f) = n$ if and only if
$Z(f)$ has an irreducible component of codimension $n$.
Thus, we have the following corollary of Theorem~\ref{thm:vasc}, which was proven by Hickel, \cite{Hic},
in the analytic setting. It is not too hard to show that this is in fact equivalent to Theorem~\ref{thm:vasc} when $m = n$.

\begin{cor} \label{corhickel}
    Let $f = (f_1,\dots,f_n)$ be a tuple of germs of holomorphic functions in $\Ok_{\Cn,0}$ vanishing at $\{0\}$, and let
    $J_f$ be the Jacobian determinant of $f$. 
    Then $J_f \in \mathcal{J}(f_1,\dots,f_n)$ if and only if $\codim Z(f_1,\dots,f_n) < n$.
    In addition, if \linebreak $\codim Z(f_1,\dots,f_n) = n$, then $\mathfrak{m} J_f \subseteq \mathcal{J}(f_1,\dots,f_n)$.
\end{cor}

We will use the generalization Theorem~\ref{thmRcomparison} of Theorem~\ref{thmRcomparisonIdeals}
to give a proof of this theorem by means of residue currents, the proof is given in
Section~\ref{secthickel}.

The results in \cite{Hic} concern more general rings than just $\Ok = \Ok_{\Cn,0}$, the ring
of germs of holomorphic functions.
In the proof in \cite{Hic}, as is the case here, residues are used. However,
the proof in \cite{Hic} uses Lipman residues, which are very much algebraic
in nature, compared to Andersson-Wulcan currents, which are analytic in nature.

In the other applications of our comparison formula that we consider in the introduction
we consider Andersson-Wulcan currents associated to Hermitian resolutions.
In the proof of Theorem~\ref{thm:vasc} we use the comparison formula when the
source complex is the Koszul complex of $f$, which is generically exact, and exact
if and only if $f$ is a complete intersection.
The target complex is a free resolution of the ideal
$\mathcal{J}(f)$, and in order to get the induced morphism between the complexes, it
is only required that the target complex is exact, see Proposition~\ref{propcomplexcomparison}.

The current associated to the Koszul complex of $f$ is called the Bochner-Martinelli current, as
introduced in \cite{PTY}. In fact, Corollary~\ref{corhickel} was an important tool in the study
of annihilators of Bochner-Martinelli currents in \cite{JW}.

\subsection{Residue currents with prescribed annihilator ideals on analytic varieties} \label{ssect:intro-prescribed}

One of the main applications when constructing the comparison formula was to construct
residue currents with prescribed annihilator ideals on singular varieties,
generalizing the construction of Andersson-Wulcan. Let $\mathcal{J} \subseteq \Ok_Z$ be an ideal
on an analytic variety $Z \subseteq \C^n$. If one considers the maximal lifting $\mathcal{J} + \mathcal{I}_Z$
of $\mathcal{J}$ to an ideal in $\Ok_{\C^n}$, then the Andersson-Wulcan current $R^{\mathcal{J}+\mathcal{I}_Z} \wedge dz$
is a current on $\C^n$ whose annihilator is $\mathcal{J}+\mathcal{I}_Z$. Since the annihilator
contains $\mathcal{I}_Z$, this current is annihilated by all holomorphic functions vanishing at $Z$, and one gets a well-defined
multiplication of this current with $\Ok_Z$. Since the annihilator as a $\Ok_{\C^n}$-module is $\mathcal{J} + \mathcal{I}_Z$,
its annihilator as a $\Ok_Z$-module is $\mathcal{J}$. We have thus constructed a current with a prescribed annihilator
on a singular subvariety of $\C^n$.
A priori, this current is just a current on $\C^n$. It would be more satisfactory
that it defines an intrinsic current on $Z$, which means that it is annihilated by all smooth forms vanishing on $Z$.
This is indeed the case, and in \cite{Lar4} we prove this using the comparison formula, give this construction a
more intrinsic interpretation, and show that this construction indeed generalizes the construction of Andersson-Wulcan when
the variety is smooth.

Trying to prove that $R^{\mathcal{J} + \mathcal{I}_Z} \wedge dz$ is a current on $Z$ was actually how we were lead
to discover the comparison formula. To prove that $R^{\mathcal{I}_Z}\wedge dz$ corresponds to a current on $Z$
is rather straightforward, using properties of pseudomeromorphic currents if $Z$ has pure dimension.
Since the holomorphic annihilator of $R^{\mathcal{J}+\mathcal{I}_Z}$ is larger than
that of $R^{\mathcal{I}_Z}$, and it has smaller support, it should be easier to annihilate
it, and hence $R^{\mathcal{J}+\mathcal{I}_Z}\wedge dz$ should also correspond to a current on $Z$.
One way of making this into a formal mathematical argument would be to express
$R^{\mathcal{J}+\mathcal{I}_Z}$ in terms of $R^{\mathcal{I}_Z}$.
In the case of two complete intersections $f$ and $g$ instead of $\mathcal{J}+\mathcal{I}_Z$
and $\mathcal{I}_Z$, the transformation law expresses this relation.
Trying to extend this to more general ideals, we arrived at Theorem~\ref{thmRcomparisonIdeals}.

More precisely, by Theorem~\ref{thmRcomparisonIdeals}, we can write
\begin{equation} \label{eq:RJI}
    R^{\mathcal{J}+\mathcal{I}_Z} \wedge dz = a R^{\mathcal{I}_Z} \wedge dz + \nabla M \wedge dz,
\end{equation}
and it thus remains to prove that $\nabla M \wedge dz$ is annihilated by any smooth form vanishing on $Z$.
This can be proven by induction, reducing to the fact that $a R^{\mathcal{I}_Z} \wedge dz$ is a current on $Z$.
In fact, in \cite{Lar4} we prove something stronger, namely, we express \eqref{eq:RJI} as the push-forward of the current
\begin{equation*}
    a \omega_Z + \nabla (V^E \wedge \omega_Z)
\end{equation*}
on $Z$, where $V^E$ and $\omega_Z$ are explicit almost semi-meromorphic currents on $Z$.

\section*{Acknowledgements}

I would like to thank Mats Andersson and Elizabeth Wulcan for valuable
discussions in the preparation of this article.

\section{Andersson-Wulcan currents and pseudomeromorphic currents}\label{sectawcurrents}

In this section we recall the construction of residue currents associated to
Hermitian resolutions of ideals, or more generally, residue currents associated to
generically exact Hermitian complexes, as constructed in \cite{AW1} and \cite{AndInt2}.
This is done in a rather detailed manner, since in order to prove the
comparison formula and the properties of the currents appearing in the formula,
we require rather detailed knowledge of the construction of Andersson-Wulcan
currents and their properties.

Let $(E,\varphi)$ be a \emph{Hermitian complex} (i.e., a complex of free $\Ok$-modules,
such that the corresponding vector bundles are equipped with Hermitian metrics),
which is generically exact, i.e., the complex is pointwise exact outside some analytic set $Z$ of positive codimension.
Mainly, $(E,\varphi)$ will be a free resolution of a module $\Ok/\mathcal{J}$,
for some ideal $\mathcal{J} \subseteq \Ok$.
When we refer to exactness of the complex, we mean that the induced complex
of sheaves of $\Ok$-modules is exact. When we refer to exactness as vector bundles,
we will refer to it as pointwise exactness. This is in contrast to the notation
in for example \cite{AW1} where the induced complex of sheaves of $\Ok$-modules is
denoted $\Ok(E)$, and exactness as vector bundles or sheaves depends on if the complex
is referred to as $E$ or $\Ok(E)$.

\subsection{The superbundle structure of the total bundle $E$} \label{ssectsuper}

The bundle $E = \oplus E_k$ has a natural superbundle structure, i.e., a $\Z_2$-grading,
which splits $E$ into odd and even elements $E^+$ and $E^-$, where $E^+ = \oplus E_{2k}$
and $E^- = \oplus E_{2k+1}$. Then $\D'(E)$, the sheaf of current-valued sections
of $E$, inherits a superbundle structure by letting the degree of an element $\mu \otimes \omega$
be the sum of the degrees of $\mu$ and $\omega$ modulo 2, where $\mu$ is a current and $\omega$
is a section of $E$.

The bundle $\End E$ also inherits a superbundle structure by letting the even elements be
the endomorphisms preserving the degree, and the odd elements the endomorphisms
switching the degree.
Given $g$ in $\End E$, we consider it also as an element of $\End \D'(E)$
by the formula
\begin{equation*}
    g (\mu \otimes \omega) = (-1)^{(\deg g)(\deg \mu)}\mu\otimes g\omega
\end{equation*}
if $g$ is homogeneous.
We also consider $\dbar$ as acting on $\D'(E)$ by the
formula $\dbar (\mu\otimes\omega) = \dbar \mu \otimes \omega$ if $\omega$ is a holomorphic
section of $E$.

We let $\nabla := \varphi - \dbar$. Note that the action of $\varphi$ on $\D'(E)$
is defined so that $\dbar$ and $\varphi$ anti-commute, and hence $\nabla^2 = 0$.
Note also that since $\varphi$ and $\dbar$ are odd, $\nabla$ is odd.

The $\Ok$-morphism $\nabla$ induces an $\Ok$-morphism $\nabla_\End$ on $\D'(\End E)$ by the formula
\begin{equation} \label{eqnablaenddef}
    \nabla(\alpha \xi) = \nabla_\End(\alpha) \xi + (-1)^{\deg \alpha} \alpha \nabla \xi,
\end{equation}
where $\alpha$ is a section of $\D'(\End E)$ and $\xi$ is a section of $E$.
By the fact that $\nabla^2 = 0$, and that $\nabla$ is odd, we also get that $\nabla_\End^2 = 0$.
Note also that if $\alpha$ and $\beta$ are sections of $\D'(\End E)$ of which at least
one of them is smooth, so that $\alpha \beta$ is defined, then
\begin{equation} \label{eqnablaleibniz}
    \nabla_\End(\alpha \beta) = \nabla_\End(\alpha) \beta + (-1)^{\deg \alpha} \alpha \nabla_\End \beta.
\end{equation}

\subsection{Pseudomeromorphic currents}

Many arguments regarding Andersson-Wulcan currents use the fact that they
are pseudomeromorphic. Pseudomeromorphic currents were introduced in \cite{AW2},
based on similarities in the construction of Andersson-Wulcan currents and
Coleff-Herrera products.

A current of the form
\begin{equation*}
        \frac{1}{z_{i_1}^{n_1}}\cdots\frac{1}{z_{i_k}^{n_k}}\dbar\frac{1}{z_{i_{k+1}}^{n_{k+1}}}\wedge\cdots\wedge\dbar\frac{1}{z_{i_m}^{n_m}}\wedge \alpha
\end{equation*}
in some local coordinate system $z$, where $\alpha$ is a smooth form with compact support, is said to be an
\emph{elementary current}. A current on a complex manifold $X$ is said to be
\emph{pseudomeromorphic}, denoted $T \in \PM(X)$, if it can be written as a locally
finite sum of push-forwards of elementary currents under compositions of modifications
and open inclusions.
As can be seen from the construction,
Coleff-Herrera products, Andersson-Wulcan currents and all currents appearing in this
article are pseudomeromorphic. In addition, as is apparent from the definition, the class
of pseudomeromorphic currents is closed under push-forwards of currents under modifications
and under multiplication by smooth forms.

An important property of pseudomeromorphic currents is that they satisfy the following
\emph{dimension principle},  \cite{AW2}*{Corollary~2.4}.
\begin{prop} \label{proppmdim}
    If $T \in \PM(X)$ is a $(p,q)$-current with support on a variety $Z$,
    and $\codim Z > q$, then $T = 0$.
\end{prop}

Another important property is the following, \cite{AW2}*{Proposition~2.3}.
\begin{prop} \label{proppmantiholo}
    If $T \in \PM(X)$, and $\Psi$ is a holomorphic form vanishing on $\supp T$,
    then
    \begin{equation*}
        \overline{\Psi} \wedge T = 0.
    \end{equation*}
\end{prop}

Pseudomeromorphic currents also have natural restrictions to analytic subvarieties.
If $T \in \PM(X)$, $Z \subseteq X$ is a subvariety of $X$, and $h$ is a tuple of
holomorphic functions such that $Z = Z(h)$, one can define
\begin{equation*}
    {\bf 1}_{X\setminus Z} T := \lim_{\epsilon\to 0^+} \chi(|h|^2/\epsilon) T \text{ and } {\bf 1}_Z T := T - {\bf 1}_{X\setminus Z} T.
\end{equation*}
This definition is independent of the choice of tuple $h$, and ${\bf 1}_Z T$ is a pseudomeromorphic
current with support on $Z$.

\subsection{Almost semi-meromorphic currents} \label{ssectas}

Let $f$ be a holomorphic function on $X$, or, more generally, a holomorphic section of a line bundle over $X$.
The associated \emph{principal value current} $1/f$ can be defined, e.g., as the limit 
\begin{equation*}
\lim_{\epsilon\to 0^+}\chi(|f|^2/\epsilon)\frac{1}{f}, 
\end{equation*}
where as before, $\chi$ is a smooth cut-off function.

A \emph{semi-meromorphic current} is a current of the form $\omega/f$ where $\omega$ is a smooth form.
Following \cite{AS2}, we say that a (pseudomeromorphic) current $A$ is \emph{almost semi-meromorphic}, $A\in ASM(X)$, 
if there is a modification $\pi:X'\to X$ such that $A=\pi_*(\omega/f)$ where $f$ is a holomorphic section of a line
bundle $L\to X'$ that does not vanish identically on $X'$ and $\omega$ is a smooth form with values in $L$.

By the dimension principle, a semi-meromorphic current has the SEP, and it then follows that almost
semi-meromorphic currents have the SEP as well. In particular, if a smooth form $\alpha$, a priori
defined outside a subvariety $W\subset X$, has an extension as a current $A\in ASM(X)$, then $A$ is
unique. Moreover, $A=\lim_{\epsilon\to 0^+} \chi (|h|^2/\epsilon) \alpha$, where $h \not\equiv 0$ is any tuple of
holomorphic functions that vanishes on $W$. We will sometimes be sloppy and use the same notation
for the smooth form $\alpha$ and its extension. 

It follows from the definition that $A\in ASM(X)$ is smooth outside a
proper subvariety of $X$. Following \cite{AWPM2}, 
we let the \emph{Zariski singular support} of $a$ be the smallest
Zariski-closed set $W$ such that $A$ is smooth outside $W$. 
If $A, B\in ASM(X)$, there is a unique current $A\wedge B\in ASM(X)$
that coincides with the smooth form $A\wedge B$ outside the Zariski singular
supports of $A$ and $B$.

Assume that $A\in ASM(X)$ has Zariski singular support $W$. Then one can write 
\begin{equation*}
\dbar A = B + R(A),
\end{equation*}
where $B={\bf 1}_{X\setminus W} \dbar A$ is the almost semi-meromorphic
continuation of $\dbar A$, and $R(A)={\bf 1}_W\dbar A$ is the
\emph{residue} of $A$, see \cite{AWPM2}*{Section~4.1}. Note that
$\dbar(1/f)=R(1/f)$.  
If $A$ is the principal value current
$A=\lim_{\epsilon\to 0^+}\chi(|h|^2/\epsilon)\alpha$, then 
$R(A)=\lim_{\epsilon\to 0^+}\dbar\chi(|h|^2/\epsilon)\wedge \alpha$. 
We also notice that if $\omega$ is smooth, then 
\begin{equation}\label{eq:Romega}
R(\omega\wedge A)=(-1)^{\deg \omega} \omega \wedge R(A). 
\end{equation}

If $(E,\varphi)$ is a complex of free $\Ok$-modules, and 
$A$ and $B$ are almost semi-meromorphic $\End(E)$-valued currents such that
$\nabla_{\End} A = B$ where $A$ and $B$ are smooth, then
\begin{equation} \label{eq:Rnabla}
    R(A) = B - \nabla_{\End} A,
\end{equation}
which follows since $\dbar A = \varphi_{\End} A - B$ where $A$ and $B$ are smooth,
and $\varphi_{\End} A - B$ has an extension as a semi-meromorphic current,
so $R(A) = \dbar A - (\varphi_{\End} A - B)$, which gives \eqref{eq:Rnabla}.

\subsection{The residue current $R$ associated to a generically exact Hermitian complex} \label{ssectrescur}

Let $Z$ be the set where $(E,\varphi)$ is not pointwise exact. Outside of $Z$,
let $\sigma_k^E : E_{k-1} \to E_k$ be the right-inverse to $\varphi_k$
which is minimal with respect to the metrics on $E$,
i.e., $\varphi_k \sigma_k^E|_{\Im \varphi_k} = \Id_{\Im \varphi_k}$, $\sigma_k^E = 0$ on
$(\Im \varphi_k)^\perp$, and $\Im \sigma_k^E \perp \ker \varphi_k$.
Then,
\begin{equation} \label{eqfsigmasigmaf}
    \varphi_{k+1}\sigma_{k+1}^E + \sigma_k^E \varphi_k = \Id_{E_k}.
\end{equation}
From \cite{AW1} it follows that if $\sigma^E := \sum \sigma_k^E$, then
\begin{equation*}
    u^E := \sum_{k=1}^N \sigma^E(\dbar\sigma^E)^{k-1}
\end{equation*}
has an extension $U^E$ as a current in $ASM(X)$.
From \eqref{eqfsigmasigmaf} it follows that $\nabla_{\End} u^E = \Id_E$ outside of $Z$.
The residue current $R^E$ can then be defined as the residue of $U^E$,
\begin{equation*}
    R^E := R(U^E).
\end{equation*}
Using that $\nabla_{\End} u^E = \Id_E$ outside of $Z$, by \eqref{eq:Rnabla},
\begin{equation*}
    R^E = I_E - \nabla_{\End} U^E,
\end{equation*}
which is the original definition of $R^E$ from \cite{AW1}. From this definition it is clear
that $\nabla_{\End} R^E = 0$.
The current $R^E$ satisfies the fundamental property that if $E$ is a free resolution of $\Ok/\mathcal{J}$,
then $\ann_{\Ok} R^E = \mathcal{J}$.

Since $R^E$ is a $\End(E)$-valued current, it consists of various components $R^\ell_k$,
where $R^\ell_k$ is the part of $R^E$ taking values in $\Hom(E_\ell,E_k)$ and $R^\ell_k$ is a $(0,k-\ell)$-current.
In case we know more about the complex $E$, more can be said about which components $R^\ell_k$
are non-vanishing. First, if $Z$ is the set where $E$ is not pointwise exact, since $R^\ell_k$
is a pseudomeromorphic $(0,k-\ell)$-current with support in $Z$, 
\begin{equation*}
    R^\ell_k = 0 \text{ if $k-\ell < \codim Z$ }.
\end{equation*}
If $E$ is exact, i.e., a free resolution, then $R^\ell_k = 0$ if $\ell \neq 0$, \cite{AW1}*{Theorem~3.1}.
We thus get that if $E$ is a free resolution of length $N$ of $\Ok/\mathcal{J}$, and $p = \codim Z(\mathcal{J})$, then
\begin{equation*}
    R^E = \sum_{k = p}^N R^0_k.
\end{equation*}

\subsection{Residue currents associated to the Koszul complex} \label{ssectbm}

Let $f = (f_1,\dots,f_p)$ be a tuple of holomorphic functions.
Then there exists a well-known complex associated to $f$, the Koszul
complex $(\bigwedge^k \Ok^{\oplus p},\delta_f)$ of $f$, which is pointwise exact outside
of the zero set $Z(f)$ of $f$. We let $e_1,\dots,e_p$ be the trivial frame of $\Ok^{\oplus p}$,
and identify $f$ with the section $f = \sum f_i e_i^*$ of $(\Ok^{\oplus p})^*$,
so that $\delta_f$ is the contraction with $f$.

In \cite{PTY} Passare, Tsikh and Yger defined the \emph{Bochner-Martinelli current} of a tuple $f$,
which we will denote by $R^f$.
One way of defining it is as the Andersson-Wulcan current associated
to the Koszul complex of $f$, see \cite{AndIdeals} for a presentation from this viewpoint.

In case the tuple $f$ defines a complete intersection, the Koszul complex of $f$
is exact, i.e., a free resolution of $\Ok/\mathcal{J}(f)$, so the annihilator of
the Bochner-Martinelli current equals $\mathcal{J}(f)$.
Another current with the same annihilator is the Coleff-Herrera product of $f$,
\eqref{comp:eqchp}, which can be defined for examples as
\begin{equation*}
    \dbar\frac{1}{f_p}\wedge\dots\wedge \dbar\frac{1}{f_1}  := \lim_{\epsilon\to 0^+} \frac{\dbar \chi(|f_p|^2/\epsilon)}{f_p} \wedge\dots\wedge \frac{\dbar \chi(|f_1|^2/\epsilon)}{f_1}.
\end{equation*}
In fact, these two currents coincide.

\begin{thm} \label{thmbmch}
    Let $f = (f_1,\dots,f_p)$ be a tuple of holomorphic functions defining a complete intersection.
    Let $R^f$ be the Bochner-Martinelli current of $f$, $R^f = \mu \wedge e_1\wedge \dots \wedge e_p$,
    and let $\mu^f$ be the Coleff-Herrera product of $f$. Then, $\mu = \mu^f$.
\end{thm}

The theorem was originally proved in \cite{PTY}*{Theorem~4.1}. See also \cite{AndCH}*{Corollary~3.2}
for an alternative proof.

\subsection{Coleff-Herrera currents} \label{ssectchcurr}

Coleff-Herrera currents (in contrast to Coleff-Herrera \emph{products} as discussed above)
were introduced in \cite{DS} (under the name ``locally residual currents''),
as canonical representatives of cohomology classes in moderate local cohomology.
Let $Z$ be a subvariety of pure codimension $p$ of a complex manifold $X$.
A $(*,p)$-current $\mu$ on $X$ is a \emph{Coleff-Herrera current}, denoted $\mu \in CH_Z$,
if $\dbar \mu = 0$, $\overline{\psi} \mu = 0$ for all holomorphic functions $\psi$ vanishing
on $Z$, and $\mu$ has the standard extension property, SEP, with respect to $Z$, i.e.,
${\bf 1}_V \mu = 0$ for any hypersurface $V$ of $Z$.

This description of Coleff-Herrera currents is due to Bj\"ork, see \cite{BjDmod}*{Chapter~3}
and \cite{BjAbel}*{Section~6.2}. In \cite{DS} locally residual currents were defined
as currents of the form $\omega \wedge R^{h}$, where $\omega$ is a holomorphic
$(*,0)$-form, and $Z = Z(h)$ (at least if $Z$ is a complete intersection defined by $h$).

One particular case of Coleff-Herrera currents that will be of interest to us
are Andersson-Wulcan currents $R^E$ associated to free resolutions
$(E,\varphi)$ of minimal length of Cohen-Macaulay modules $\Ok/\mathcal{J}$.
Such a current is $\dbar$-closed since $\nabla R^E = 0$ implies that
$\dbar R^E_p = \varphi_{p+1} R^E_{p+1} = 0$ since $E$ is assumed to be of minimal
length. The other properties needed in order to be a Coleff-Herrera current
are satisfied by the fact that they are pseudomeromorphic, Proposition~\ref{proppmdim}
and Proposition~\ref{proppmantiholo}.

\subsection{Singularity subvarieties of free resolutions} \label{ssectbef}

In the study of residue currents associated to finitely generated $\Ok$-modules
an important ingredient is certain singularity subvarieties associated to the module.
Given a free resolution $(E,\varphi)$ of a finitely generated module $G$, the variety $Z^E_k$ is defined
as the set where $\varphi_k$ does not have optimal rank. These sets are independent of the choice of free
resolution. Note that these varieties can equally well be defined for any complex of free $\Ok$-modules
$(E,\varphi)$ which is generically exact.

The fact that these sets are important in the study of residue currents associated to generically exact
Hermitian complexes stems from the following. Outside of $Z_k^E$ the
form $\sigma_k^E$ defined in Section~\ref{ssectrescur} is smooth, so by using that
$\sigma_{l+1}^E \dbar \sigma_l^E = \dbar \sigma_{l+1}^E \sigma_l^E$ (see \cite{AW1}*{(2.3)}),
$R_k^E  = \dbar \sigma_k^E R_{k-1}^E$ outside of $Z_k^E$. This combined with the dimension
principle for pseudomeromorphic currents allows for inductive arguments regarding
residue currents.

If $\codim G = p$, then $Z_k^E = \supp G$ for $k \leq p$, \cite{Eis}*{Corollary~20.12}.
In addition, by \cite{Eis}*{Theorem~20.9},
\begin{equation} \label{eq:codimZkk}
    \codim Z^E_k \geq k.
\end{equation}
In particular,
\begin{equation} \label{eq:codimZE}
    \codim Z^E_k \geq \codim G.
\end{equation}
In fact, \cite{Eis}*{Theorem~20.9} is a characterization of exactness, the
\emph{Buchsbaum-Eisenbud criterion}, which says that a generically exact complex $(E,\varphi)$ of free
modules is exact if and only if $\codim Z_k^E \geq k$.

\section{A comparison formula for Andersson-Wulcan currents} \label{sectcomparisonformula}

The starting point of Theorem~\ref{thmRcomparisonIdeals} is that when $\mathcal{I} \subseteq \mathcal{J}$,
the natural surjection $\pi : \Ok/\mathcal{I} \to \Ok/\mathcal{J}$
induces a morphism of complexes $a : (F,\psi) \to (E,\varphi)$, where $(F,\psi)$ and 
$(E,\varphi)$ are free resolutions of $\Ok/\mathcal{I}$ and $\Ok/\mathcal{J}$ respectively.
The existence of such a morphism holds much more generally in homological algebra,
of which the following formulation is suitable for our purposes. This is sometimes
referred to as the \emph{comparison theorem}.

\begin{prop} \label{propcomplexcomparison}
    Let $\alpha : G \to H$ be a homomorphism of $\Ok$-modules, let $(F,\psi)$ be a complex of
    free $\Ok$-modules with $\coker \psi_1 = G$, and let $(E,\varphi)$ be a free resolution of $H$.
    Then, there exists a morphism $a : (F,\psi) \to (E,\varphi)$ of complexes which extends $\alpha$.
    If $\tilde{a}$ is any other such morphism, then there exists a homotopy $s : (F,\psi) \to (E,\varphi)$
    of degree $-1$ such that $a_i - \tilde{a}_i = \varphi_{i+1} s_i - s_{i-1} \psi_i$.
\end{prop}

That $a$ extends $\alpha$ means that the map induced by $a_0$ on
$F_0/(\im \psi_1) \cong G \to H \cong E_0/(\im \varphi_1)$ equals $\alpha$.
Both the existence and uniqueness up to homotopy of $a$ follows from defining $a$ or
$s$ inductively by a relatively straightforward diagram chase,
see \cite{Eis}*{Proposition~A3.13}.

\medskip

This is the general formulation of our main theorem, Theorem~\ref{thmRcomparisonIdeals}.

\begin{thm} \label{thmRcomparison}
Let $a : (F,\psi) \to (E,\varphi)$ be a morphism of generically exact Hermitian complexes, and
let $M' := U^E a U^F$ be the product of the almost semi-meromorphic currents $U^E$ and $a U^F$.
Let $M$ be the residue
\begin{equation} \label{eqmdef}
    M := R(U^E a U^F)
\end{equation}
Then
\begin{equation} \label{eqRcomparison2}
    R^E a - a R^F = \nabla_{\End} M,
\end{equation}
where $\nabla_{\End}$ acts on the complex $(E\oplus F,\varphi \oplus \psi)$.
\end{thm}

By definition of the residue, if $h$ is a tuple of holomorphic functions
such that $h \not\equiv 0$, and $Z(h)$ contains the set where $(E,\varphi)$ and $(F,\psi)$ are not pointwise exact, then
\begin{equation*}
  M = R(U^E a U^F) = \lim_{\epsilon\to 0^+} \dbar\chi(|h|^2/\epsilon) \wedge U^E a U^F.
\end{equation*}

Note that $\nabla_{\End}$ is defined with respect to the complex 
$(E \oplus F, \varphi \oplus \psi)$, and the superstructure, as in
Section~\ref{ssectsuper}, of this complex is the grading
$(E\oplus F)^+ = E^+ \oplus F^+$, $(E\oplus F)^- = E^- \oplus F^-$.

If we let $M^\ell_k$ be the part of $M$ in \eqref{eqmdef} with values in $\Hom(F_\ell,E_k)$, we get from \eqref{eqnablaenddef}
and \eqref{eqnablaleibniz} that
\begin{equation} \label{eqRcomparison-component-gen}
    (R^E)^\ell_k a_\ell - a_k (R^F)^\ell_k  = \varphi_{k+1} M^\ell_{k+1} + M^{\ell-1}_k \psi_\ell - \dbar M^\ell_k. 
\end{equation}
In the important case $\ell = 0$, if we write $M_k$ for the $\Hom(F_0,E_k)$-valued part of $M$, and
$R^E_k$ and $R^F_k$ for the $\Hom(E_0,E_k)$- and $\Hom(F_0,F_k)$-valued parts of $R^E$ and $R^F$, we get
\begin{equation} \label{eqRcomparison-component-zero}
    R^E_k a_0 - a_k R^F_k = \varphi_{k+1} M_{k+1} - \dbar M_k.
\end{equation}

\begin{proof}
    Since $a$ is a morphism of complexes, $\varphi a = a \psi$, and hence $\nabla_{\End} a = \varphi a - a \psi = 0$.
    Let $Z$ be a variety containing the sets where $(E,\varphi)$ and $(F,\psi)$ are not pointwise exact.
    Since outside of $Z$, $\nabla_{\End} U^E = \Id_E$ and $\nabla_{\End} U^F = \Id_F$, we get using
    \eqref{eqnablaleibniz} and the fact that $U^E$ has odd degree and $a$ has even degree that
    \begin{equation*}
        \nabla_{\End} M' = a U^F - U^E a
    \end{equation*}
    outside of $Z$. Since $M'$, $a U^F$ and $U^E a$ are almost semi-meromorphic,
    \begin{equation*}
        M = R(U^E a U^F) = a U^F - U^E a - \nabla_{\End} M'
    \end{equation*}
    by \eqref{eq:Rnabla}.
    Applying $\nabla_{\End}$ to this equation we get \eqref{eqRcomparison2} since
    $\nabla_{\End}^2 = 0$, and
    \begin{align*}
        \nabla_{\End} (a U^F - U^E a) &= a \nabla_{\End} U^F - \nabla_{\End} U^E a = \\
        &= a (\Id_{F} - R^F) - (\Id_{E} - R^E) a = R^E a - a R^F.
    \end{align*}
\end{proof}

The main idea in the proof of Theorem~\ref{thmRcomparison}, to form a $\nabla$-potential
to $R - R'$, essentially of the form $\nabla(U\wedge U')$, appears in various works regarding residue
currents. One example is in \cite{AndIdeals} and \cite{AW1} where this idea is used
to prove that under suitable conditions the residue currents do
not depend on the choice of metrics. This corresponds to applying the comparison formula in the
case when $(E,\varphi)$ and $(F,\psi)$ have the same underlying complex, but
are equipped with different metrics.

Another instance where such a construction appears is in \cite{Lar},
regarding the transformation law for Coleff-Herrera products of (weakly) holomorphic functions,
of which its relation to the comparison formula is elaborated in Remark~\ref{remtransl}.
It also appears in \cite{AndResCrit} and \cite{W1}, regarding products of residue currents,
but the relation to the comparison formula is not as apparent.

\begin{remark}
    Note that in Proposition~\ref{propcomplexcomparison} the complex $(F,\psi)$ does
    not have to be exact. For our comparison formula to work, neither the complex
    $(E,\varphi)$ has to be exact, as long as the morphism $a$ exists.
    For example, if we have $f = g A$ for some tuples $g$ and $f$ of holomorphic functions,
    and a holomorphic matrix $A$ as in Remark~\ref{remtransl}, then $A$ induces a morphism
    between the Koszul complexes of $f$ and $g$. We can then apply the comparison formula
    also when the Koszul complex of $g$ is not exact.
\end{remark}

\subsection{The current $M$}

We will here describe the current $M$ a bit more thoroughly. First of all,
we have the following inductive description.

\begin{lma} \label{lma:Mellk-induction}
    Let $(E,\varphi)$, $(F,\psi)$, $a : (F,\psi) \to (E,\varphi)$ and $M$ be as in Theorem~\ref{thmRcomparison},
    and let $M^\ell_k$ be the part of $M$ which takes values in $\Hom(F_\ell,E_k)$.
    Then, outside of $Z^E_k$, where $\sigma^E_k$ is smooth,
    \begin{equation} \label{eq:Mellk-induction}
        M^\ell_k = \dbar \sigma^E_k M^\ell_{k-1} - \sigma^E_k a_{k-1} (R^F)^\ell_{k-1}.
    \end{equation}
\end{lma}

\begin{proof}
    Using that $\sigma^E_{j+1} \dbar\sigma^E_j = \dbar\sigma^E_{j+1}\sigma^E_j$, one gets that
    \begin{equation*}
        \sigma^E_k\dbar\sigma^E_{k-1}\cdots\dbar\sigma^E_{m+1} = \dbar\sigma^E_k \cdots\dbar\sigma^E_{m+2}\sigma^E_{m+1}.
    \end{equation*}
    Hence,
    \begin{equation*}
        M^\ell_k = \sum_{m=\ell+1}^{k-1} R(\dbar\sigma^E_k \dbar\sigma^E_{k-1} \cdots \dbar\sigma^E_{m+2} \sigma^E_{m+1} a_m
        \sigma^F_m \dbar\sigma^F_{m-1} \cdots \dbar\sigma^F_\ell).
    \end{equation*}
    Splitting the sum into when $\ell+1 \leq m \leq k-2$ and when $m = k-1$, and using \eqref{eq:Romega}, we get 
    \begin{align*}
        M^\ell_k = \dbar\sigma^E_k \sum_{m=\ell+1}^{k-2} R(\dbar\sigma^E_{k-1} \dbar\sigma^E_{k-2} \cdots \dbar\sigma^E_{m+2} \sigma^E_{m+1} a_m
        \sigma^F_m \dbar\sigma^F_{m-1} \cdots \dbar\sigma^F_\ell) \\
        - \sigma^E_k a_{k-1} R(\sigma^F_{k-1} \dbar\sigma^F_{k-2} \cdots \dbar\sigma^F_\ell) = \dbar \sigma^E_k M^\ell_{k-1} -
        \sigma^E_k a_{k-1} (R^F)^\ell_{k-1}.
    \end{align*}
\end{proof}

In order to understand when parts of the current $M$ in Theorem~\ref{thmRcomparison} vanishes,
we begin with the following lemma about when parts of the current $R^F$ vanishes.

\begin{lma} \label{lma:Rvanishing}
    Let $(F,\psi)$ be a generically exact Hermitian complex,
    and assume that $\codim Z^F_{\ell+m} \geq m+1$ for $m=1,\dots,k-\ell$.
    Then $(R^F)^\ell_k = 0$, where $(R^F)^\ell_k$ is the part of $R^F$ with values in $\Hom(F_\ell,F_k)$.
\end{lma}

In the special case when $(F,\psi)$ is a free resolution and $\ell \geq 1$, then
$\codim Z^F_{\ell+m} \geq \ell+m \geq m+1$, see \eqref{eq:codimZkk}.
The lemma thus implies that
\begin{equation} \label{eq:Rvanishing}
    (R^F)^\ell_k = 0 \text{ for $\ell \geq 1$}
\end{equation}
under these assumptions, which is \cite{AW1}*{Theorem~3.1}. The proof of Lemma~\ref{lma:Rvanishing} is the same as the proof
of \cite{AW1}*{Theorem~3.1}, as it only uses these inequalities about the codimension of the
sets $Z^F_{\ell+m}$ (and the ``vague principle'' about vanishing of residue currents
referred to in the proof was later formalized as the dimension principle, Proposition~\ref{proppmdim}).

\begin{prop} \label{prop:Mvanishing}
    Let $(E,\varphi)$, $(F,\psi)$, $a : (F,\psi) \to (E,\varphi)$ and $M$ be as in Theorem~\ref{thmRcomparison},
    and let $M^\ell_k$ be the part of $M$ which takes values in $\Hom(F_\ell,E_k)$.
    If
    \begin{align}
        \label{eq:ZF-codim} & \codim Z^F_{\ell+m} \geq m+1 \text{ for $m=1,\dots,k-\ell-1$ and} \\
        \label{eq:ZE-codim} & \codim Z^E_{\ell+m} \geq m \text{ for $m=2,\dots,k-\ell$,}
    \end{align}
    then $M^\ell_k = 0$.
\end{prop}

\begin{proof}
    We prove this by induction over $k-\ell$, starting with the first non-trivial case $k = \ell+2$.
    Since $M^\ell_{\ell+2}= R(\sigma^E_{\ell+2} a_{\ell+1}\sigma^F_{\ell+1})$ has support where
    $\sigma^E_{\ell+2}$ and $\sigma^F_{\ell+1}$ are not smooth, $\supp M^\ell_{\ell+2} \subseteq W  := Z^E_{\ell+2} \cup Z^F_{\ell+1}$.
    By assumption, $\codim W \geq 2$, and since $M^\ell_{\ell+2}$ is a pseudomeromorphic $(0,1)$-current, it is $0$ by the dimension principle.

    Note that the assumptions \eqref{eq:ZF-codim} imply by Lemma~\ref{lma:Rvanishing} that $(R^F)^\ell_{\ell+m} = 0$
    for $1 \leq m \leq k-\ell-1$.
    Assume now that we have proven that $M^\ell_{\ell+m-1} = 0$ for $3 \leq m \leq k-\ell$.
    Then, by \eqref{eq:Mellk-induction}, outside of $Z^E_{\ell+m}$,
    \begin{equation*}
        M^\ell_{\ell+m} = \dbar\sigma^E_{\ell+m} M^\ell_{\ell+m-1} - \sigma^E_{\ell+m} a_{\ell+m}(R^F)^\ell_{\ell+m-1}.
    \end{equation*}
    Since the currents $M^\ell_{\ell+m-1}$ and $R^\ell_{\ell+m-1}$ both vanish, we thus get that $M^\ell_{\ell+m}$
    vanishes outside of $Z^E_{\ell+m}$. Since $M^\ell_{\ell+m}$ is a pseudomeromorphic $(0,m-1)$-current
    with support in $Z^E_{\ell+m}$ of codimension $\geq m$, it is $0$ by the dimension principle.
    By induction, we thus conclude that $M^\ell_k = 0$.
\end{proof}

\begin{cor}
    Let $(E,\varphi)$, $(F,\psi)$, $a : (F,\psi) \to (E,\varphi)$ and $M$ be as in Theorem~\ref{thmRcomparison},
    and let $M^\ell_k$ be the part of $M$ which takes values in $\Hom(F_\ell,E_k)$.
    Assume that $(F,\psi)$ and $(E,\varphi)$ are free resolutions of modules $G$ and $H$ respectively.
    Then,
    \begin{equation} \label{eq:Mell-vanishing}
        M^\ell_k = 0 \text{ for $\ell=1,\dots,k-2$},
    \end{equation}
    and if $G$ and $H$ have codimension $\geq k$, then
    \begin{equation} \label{eq:Mcodim-vanishing}
        M^0_k = 0.
    \end{equation}
    In addition, for any $k$,
    \begin{equation} \label{eq:Mright-vanishing}
        M^0_k \psi_1 = 0.
    \end{equation}
\end{cor}

\begin{proof}
    By \eqref{eq:codimZkk}, for all $j \geq 1$, $\codim Z^E_j \geq j$,
    and $\codim Z^F_j \geq j$, and thus, \eqref{eq:Mell-vanishing} follows directly from Proposition~\ref{prop:Mvanishing}.
    In addition, if $k < \codim G$ and $k < \codim H$, then $\codim Z^F_j \geq \codim G$ and $\codim Z^E_j \geq \codim H$
    by \eqref{eq:codimZE}, so \eqref{eq:Mcodim-vanishing} also follows directly from Proposition~\ref{prop:Mvanishing}.

    By \eqref{eqRcomparison-component-gen},
    \begin{equation*}
        M^0_k \psi_1 = -\varphi_{k+1} M^1_{k+1} + \dbar M^1_k + (R^E)^1_k a_1 - a_k (R^F)^1_k,
    \end{equation*}
    and by \eqref{eq:Rvanishing} and \eqref{eq:Mell-vanishing}, all currents in the right-hand side vanish, so we have
    proven \eqref{eq:Mright-vanishing}.
\end{proof}

\section{A transformation law for Andersson-Wulcan currents associated to Cohen-Macaulay modules} \label{secttranscm}

In this section we state and prove the general version of our transformation law
for Andersson-Wulcan currents associated to Cohen-Macaulay modules.

\begin{thm} \label{thmtranscmgen}
    Let $G$ be a finitely generated $\Ok$-module of codimension $p$, and assume that $G$ is Cohen-Macaulay.
    Let $(E,\varphi)$ be a free resolution of $G$ of length $p$, and let $(F,\psi)$ be a generically exact
    Hermitian complex such that the set $Z$ where $(F,\psi)$ is not pointwise exact has codimension $\geq p$.
    If $a : (F,\psi) \to (E,\varphi)$ is a morphism of complexes, then
    \begin{equation*}
        R^E_p a_0 = a_p R^F_p.
    \end{equation*}
    If $a_0$ is any morphism $F_0 \to E_0$ such that $a_0(\im \psi_1) \subseteq \im \varphi_1$, then
    $a_0$ can be extended to a morphism $a : (F,\psi) \to (E,\varphi)$.
\end{thm}

Note in particular, if $F_0 \cong \Ok \cong E_0$, $a_0 : F_0 \to E_0$ is this isomorphism,
and $\mathcal{J} := \im \varphi_1$, and $\mathcal{I} := \im \psi_1$, then $a_0$ can be extended
if $\mathcal{I} \subseteq \mathcal{J}$, and the morphism $a$ then extends the natural surjection
$\pi : \Ok/\mathcal{I} \to \Ok/\mathcal{J}$.

\begin{proof}
    The last part about the existence of $a$ follows immediately from
    Proposition~\ref{propcomplexcomparison}.

    By \eqref{eqRcomparison-component-zero},
    \begin{equation*}
        R^E_p a_0 = a_p R^F_p + \varphi_{p+1} M^0_{p+1} - \dbar M^0_p.
    \end{equation*}
    Since $(E,\varphi)$ has length $p$, $\varphi_{p+1} M^0_{p+1} = 0$, and $M^0_p = 0$ by \eqref{eq:Mcodim-vanishing}.
\end{proof}

\begin{ex} \label{excmcurve}
    Let $\pi : \C \to \C^3$, $\pi(t) = (t^3,t^4,t^5)$, and let $Z$ be the germ at $0$ of $\pi(\C)$.
    One can show that the ideal of holomorphic functions vanishing at $Z$
    equals $\mathcal{J} = (y^2-xz,x^3-yz,x^2y-z^2)$.

    The module $\Ok/\mathcal{J}$ has a minimal free resolution
    \begin{equation*}
        0 \to \Ok^{\oplus 2} \xrightarrow[]{\varphi_2} \Ok^{\oplus 3} \xrightarrow[]{\varphi_1} \Ok \to \Ok/\mathcal{J},
    \end{equation*}
    where
    \begin{equation*}
        \varphi_2 = \left[ \begin{array}{cc} -z & -x^2 \\ -y & -z \\ x & y \end{array} \right]
            \text{ and }
            \varphi_1 = \left[ \begin{array}{ccc} y^2-xz & x^3-yz & x^2y-z^2 \end{array} \right].
    \end{equation*}
    To check that this is a resolution, one verifies first that it indeed is a complex.
    Secondly, since $I_1 = I(\varphi_1) = \mathcal{J}$, and $I_2 = I(\varphi_2) = \mathcal{J}$ (the
    Fitting ideals of $\varphi_1$ and $\varphi_2$), the complex is exact by the Buchsbaum-Eisenbud
    criterion, see Section~\ref{ssectbef} (and note that $Z_k^E = Z(I_k)$).

    In particular, since $\Ok/\mathcal{J}$ has a minimal free resolution of length $2$
    with $\rank E_2 = 2$, $Z$ is Cohen-Macaulay but not a complete intersection.
    However, $Z$ is in fact a set-theoretic complete intersection.
    Let $f = (z^2-x^2y,x^4+y^3-2xyz)$, and $\mathcal{I} = \mathcal{J}(f)$.
    One can verify that $Z(\mathcal{I}) = Z$, and since $\codim Z = 2$, $Z$ is indeed a set-theoretic
    complete intersection.

    Now, let $(E,\varphi)$ be the free resolution of $\Ok/\mathcal{J}$, and $(F,\psi)$ be the Koszul
    complex of $f$, which is a free resolution of $\Ok/\mathcal{I}$ since $f$ is a complete intersection.
    Since $\Ok/\mathcal{J}$ is Cohen-Macaulay and $Z(\mathcal{I}) = Z(\mathcal{J})$, we can apply
    Theorem~\ref{thmtranscm} to $(F,\psi)$ and $(E,\varphi)$. One verifies that $a : (F,\psi) \to (E,\varphi)$,
    \begin{equation*}
        a_2 = \left[ \begin{array}{c} x^3-yz  \\ y^2-xz \end{array} \right] \text{, }
        a_1 = \left[ \begin{array}{cc} 0 & y \\ 0 & x \\  -1 & 0 \end{array} \right] \text{ and }
        a_0 = \left[ \begin{array}{c} 1 \end{array} \right],
    \end{equation*}
    is a morphism of complexes extending the natural surjection $\pi : \Ok/\mathcal{I} \to \Ok/\mathcal{J}$.
    This morphism can be found with for example the computer algebra system Macaulay2.
    Since the current associated to the Koszul complex of a complete intersection $f$ is the Coleff-Herrera
    product of $f$ by Theorem~\ref{thmbmch}, we get by Theorem~\ref{thmtranscm} that
    \begin{equation*}
        R^E = \dbar\frac{1}{x^4+y^3-2xyz}\wedge\dbar\frac{1}{z^2-x^2y} \wedge
        \left[ \begin{array}{c} x^3-yz \\ y^2-xz \end{array} \right].
    \end{equation*}
\end{ex}

The fact that we can express the residue current corresponding to the ideal above
in terms of a Coleff-Herrera product can be done more generally, as the following
example shows.

\begin{ex}
Let $\mathcal{J} \subseteq \Ok$ be a Cohen-Macaulay ideal of codimension $p$,
and let $Z = Z(\mathcal{J})$. Then,  there exists a complete intersection
$(f_1,\dots,f_p)$ such that $Z \subseteq Z(f)$, see for example \cite{Lar2}*{Lemma~19}.
By the Nullstellensatz, there exist $N_i$ such that $f_i^{N_i} \in \mathcal{J}$.
Thus, by replacing $f_i$ by $f_i^{N_i}$, we can assume that $(f_1,\dots,f_p)$ is a complete
intersection such that $\mathcal{J}(f_1,\dots,f_p) \subseteq \mathcal{J}$.
Let $(F,\psi)$ be the Koszul complex of $f$, and let $(E,\varphi)$ be a free resolution of
$\Ok/\mathcal{J}$ of length $p$. By Theorem~\ref{thmtranscm}, we then have that
\begin{equation*}
    R_p^\mathcal{J} = \dbar \frac{1}{f_p}\wedge \dots \wedge \dbar \frac{1}{f_1} \wedge
    a_p(e_1\wedge\dots\wedge e_p),
\end{equation*}
where $a_p$ is the morphism in Theorem~\ref{thmtranscm}, since the current associated with the
Koszul complex of $f$ is the Coleff-Herrera product of $f$.
\end{ex}

\begin{remark} \label{remtransl}

    The transformation law for Coleff-Herrera products is a corollary of Theorem~\ref{thmtranscm}
    in the following way.
    Let $f$ and $g$ be two complete intersections of codimension $p$, and assume that there exists
    a matrix $A$ of holomorphic functions such that $f = g A$. 

Since $f$ and $g$ are complete intersections, the Koszul complexes
$(\bigwedge \Ok^{\oplus p},\delta_f)$ and $(\bigwedge \Ok^{\oplus p},\delta_g)$
are free resolutions of $\Ok/\mathcal{J}(f)$ and $\Ok/\mathcal{J}(g)$.
Since $\mathcal{J}(f) \subseteq \mathcal{J}(g)$, we get a morphism $a$ of the Koszul complexes
of $f$ and $g$ induced by the inclusion $\pi : \Ok/\mathcal{J}(f) \to \Ok/\mathcal{J}(g)$
by Proposition~\ref{propcomplexcomparison}.
In fact, the morphism $a_k : \bigwedge^k \Ok^{\oplus p} \to \bigwedge^k \Ok^{\oplus p}$
is readily verified to be
$\bigwedge^k A : \bigwedge^k \Ok^{\oplus p} \to \bigwedge^k \Ok^{\oplus p}$,
see \cite{Lar}*{Lemma~7.2}. In particular, $a_p = \bigwedge^p A = \det A$, so since
the Andersson-Wulcan currents associated to the Koszul complexes of $f$ and $g$
are the Coleff-Herrera products of $f$ and $g$, the transformation law
$\mu^g = (\det A) \mu^f$ follows directly from Theorem~\ref{thmtranscm}.

In fact, the proof of Theorem~\ref{thmtranscm} in this particular situation
becomes exactly the proof of the transformation law for Coleff-Herrera products
given in \cite{Lar}*{Theorem~7.1}.
\end{remark}

    As mentioned above, the transformation law for Coleff-Herrera products is a special
    case of Theorem~\ref{thmtranscm}. In \cite{DS2}, two proofs of the transformation law
    are given, and in fact, we can essentially use the same argument as the second proof
    of the transformation law in \cite{DS2}*{p.~54--55}, to prove Theorem~\ref{thmtranscm}.

\begin{proof}[Alternative proof of Theorem~\ref{thmtranscm}]
    Consider $E^p_\mathcal{J} := \Ext^p_{\Ok}(\Ok/\mathcal{J},\Ok)$.
    One way of computing $E^p_\mathcal{J}$ is by taking a free resolution
    $(E,\varphi)$ of $\Ok/\mathcal{J}$, applying $\Hom(\bullet, \Ok)$ and taking cohomology, i.e.,
    $E^p_\mathcal{J} \cong H^p(\Hom(E_\bullet,\Ok))$.
    On the other hand, it can also be computed by taking an injective resolution of $\Ok$,
    which can be taken as the complex of $(0,*)$-currents, $(\mathcal{C}^{0,\bullet},\dbar)$,
    applying $\Hom(\Ok/\mathcal{J},\bullet)$ to this complex, and taking cohomology,
    i.e., $E^p_\mathcal{J} \cong H^p(\Hom(\Ok/\mathcal{J},\mathcal{C}^{0,\bullet}))$.

    Since these are different realizations of $\Ext$, they are naturally isomorphic, and by
    \cite{AndNoeth}*{Theorem~1.5} this isomorphism is given by
    \begin{equation} \label{eqextisom}
        \phi : [\xi]_{H^p(\Hom(E_\bullet,\Ok))} \mapsto
        [\xi R^E_p]_{H^p(\Hom(\Ok/\mathcal{J},\mathcal{C}^{0,\bullet}))}.
    \end{equation}

    We now consider the map $\pi : \Ok/\mathcal{I} \to \Ok/\mathcal{J}$, which
    induces a map $\pi^* : E^p_\mathcal{J} \to E^p_\mathcal{I}$.
    In the first realization of $\Ext$, $\pi^*$ becomes the map
    $a_p^* : H^p(\Hom(E_\bullet,\Ok)) \to H^p(\Hom(F_\bullet,\Ok))$
    induced by $a : (F,\psi) \to (E,\varphi)$.
    In the second realization of $\Ext$, the map becomes just the identity
    map on the currents (due to the fact that currents annihilated by $\mathcal{J}$
    are also annihilated by $\mathcal{I}$).
    Thus, using the naturality of $\pi^*$ and the isomorphism \eqref{eqextisom}
    we get from the commutative diagram
\begin{equation*}
\xymatrix{
    H^p(\Hom(E_\bullet,\Ok)) \ar[r]^{\pi^*} \ar[d]^{\phi} &
    H^p(\Hom(F_\bullet,\Ok)) \ar[d]^{\phi} \\
    H^p(\Hom(\Ok/\mathcal{J},\mathcal{C}^{0,\bullet})) \ar[r]^{\pi^*} &
    H^p(\Hom(\Ok/\mathcal{I},\mathcal{C}^{0,\bullet}))
}
\end{equation*}
    that $[(a_p^*)\xi R_p^F]_{\dbar} = [\xi R_p^E]_{\dbar}$, where
    $\xi$ is a holomorphic section of $\ker \varphi_{p+1}^*$. Hence,
    $\xi a_p R^F_p = \xi R^E_p + \dbar \eta_{\xi}$, where $\eta_\xi$ is annihilated by
    $\mathcal{I}$.
    Since $(E,\varphi)$ has length $p$, $\varphi_{p+1} = 0$, so the equality holds
    for all holomorphic sections $\xi$ of $E_p$, i.e.,
    $a_p R^F_p = R^E_p + \dbar \eta$ for some (vector-valued) current $\eta$ annihilated
    by $\mathcal{I}$.
    Since $a_p$ is holomorphic and $R^F_p$ and $R^E_p$ are in $CH_Z$, see Section~\ref{ssectchcurr},
    where $Z = Z(\mathcal{I})$, we get from the decomposition
    $\ker(\mathcal{C}_Z^{0,p} \stackrel{\dbar}{\to} \mathcal{C}_Z^{0,p+1})
    = CH_Z \oplus \dbar \mathcal{C}_Z^{0,p-1}$,
    see \cite{DS2}*{Theorem~5.1}, that $\dbar \eta = 0$, where $\mathcal{C}_Z^{0,p}$ is the sheaf
    of $(0,p)$-currents supported on $Z$.

    The only difference of the proof here to the proof in \cite{DS2} is
    that we have the isomorphism \eqref{eqextisom} from \cite{AndNoeth},
    while in \cite{DS2} this isomorphism was only available if $\mathcal{J}$
    was a complete intersection ideal, see the proof of \cite{DS}*{Proposition~3.5}.
\end{proof}

We end this section with an example of how we can express Andersson-Wulcan currents associated to
Cohen-Macaulay ideals in terms of Bochner-Martinelli currents.

\begin{ex}
    Let $f = (f_1,\ldots,f_k)$ be a tuple of holomorphic functions, let $\mathcal{J} = \mathcal{J}(f_1,\dots,f_k)$
    and $Z = Z(f)$, and assume that $\codim Z = p$.
    Assume in addition that $\Ok/\mathcal{J}$ is Cohen-Macaulay.
    Note that we do not assume that $f$ is a complete intersection, i.e., that $k = p$.
    Let $\Ok^{\oplus k}$ be the trivial vector bundle with frame $e_1,\ldots,e_k$,
    and consider $f$ as a section of $(\Ok^{\oplus k})^*$, $f = \sum f_i e_i^*$.
    Let $R^f$ be the Bochner-Martinelli current associated with $f$,
    and write $R^f_p = \sum R_I \wedge e_I$, i.e., $R_I \wedge e_I$ is the component of
    $R^f_p$ with values in $e_I := e_{i_1}\wedge \dots \wedge e_{i_p} \in \bigwedge^p \Ok^{\oplus k}$.

    In \cite{AndCH} Andersson proves that if $\mu \in CH_Z$, then there exist holomorphic
    $(*,0)$-forms $\alpha_I$ such that $\mu = \sum \alpha_I \wedge R_I$ (after first
    replacing $f_i$ by $f_i^{N_i}$ such that $f_i^{N_i} \mu = 0$).
    In particular, this applies in our case to $R^\mathcal{J}$, see Section~\ref{ssectchcurr}.
    In \cite{AndCH} the $\alpha_I$ are not explicitly
    given, but when $\mu = R^\mathcal{J}$, we can obtain them from Theorem~\ref{thmtranscmgen}.
    We let $(F,\psi)$ be the Koszul complex of $f$, and $(E,\varphi)$ a minimal free
    resolution of $\Ok/\mathcal{J}$.
    Since the current associated with the Koszul complex of $f$ is the
    Bochner-Martinelli current of $f$, Theorem~\ref{thmtranscmgen} gives the factorization
    \begin{equation*}
        R^\mathcal{J} = \sum \alpha_I \wedge R_I,
    \end{equation*}
    where $\alpha_I = a_p(e_I)$.
\end{ex}

\section{A non Cohen-Macaulay example} \label{sectnoncmex}

When the ideals involved in the comparison formula are not Cohen-Macaulay,
the comparison formula does not have as simple form as in the Cohen-Macaulay
case in Section~\ref{secttranscm}. In this section we illustrate with an
example how one could still use the comparison formula also to compute the
residue current associated to a non Cohen-Macaulay ideal.

\begin{ex}
    Let $Z \subseteq \C^4$ be the variety $Z = \{ x = y = 0 \} \cup \{z = w = 0 \}$.
    The ideal $\mathcal{I}_Z$ of holomorphic functions on $\C^4$ vanishing on $Z$
    equals $\mathcal{I}_Z = \mathcal{J}(xz,xw,yz,yw)$. It can be verified that
    $\mathcal{I}_Z$ has a minimal free resolution $(E,\varphi)$ of the form
    \begin{equation*}
        0 \to \Ok \xrightarrow[]{\varphi_3} \Ok^{\oplus 4} \xrightarrow[]{\varphi_2}
        \Ok^{\oplus 4} \xrightarrow[]{\varphi_1} \Ok \to \Ok/\mathcal{I}_Z,
    \end{equation*}
    where
    \begin{align*}
        &\varphi_3 = \left[\begin{array}{c} w \\ -z \\ -y \\ x \end{array}\right] \text{, }
        \varphi_2 = \left[\begin{array}{cccc} -y & 0 & -w & 0 \\ 0 & -y & z & 0 \\
            x & 0 & 0 & -w \\ 0 & x & 0 & z \end{array}\right] \text{ and } 
        &\varphi_1 = \left[ \begin{array}{cccc} xz & xw & yz & yw \end{array} \right].
    \end{align*}
    Note that $Z$ has codimension $2$, while the free resolution above, which
    is minimal, has length $3$, so $Z$ is not Cohen-Macaulay.

We compare this resolution with the Koszul complex $(F,\psi)$ of the complete
intersection ideal $\mathcal{I} = \mathcal{J}(xz,yw)$.
One can verify that the morphism $a : (F,\psi) \to (E,\varphi)$,
\begin{equation*}
    a_2 = \frac{1}{2}\left[ \begin{array}{c} w \\ z \\ y \\ x \end{array} \right] \text{, }
    a_1 = \left[ \begin{array}{cc} 1 & 0 \\ 0 & 0 \\ 0 & 0 \\ 0 & 1 \end{array} \right] \text{ and }
    a_0 = \left[ 1 \right],
\end{equation*}
is a morphism of complexes extending the natural surjection
$\pi : \Ok/\mathcal{I} \to \Ok/\mathcal{I}_Z$ as in Proposition~\ref{propcomplexcomparison}.

By \eqref{eqRcomparison-component-zero},
$R_2^E = a_2 R_2^F + \varphi_3 M_3 - \dbar M_2$. Note that $M_2 = 0$ by \eqref{eq:Mcodim-vanishing}.
By \eqref{eq:Mellk-induction} and the fact that $M_2 = 0$,
outside of $Z_3^E = \{ 0 \}$ we get that $M_3 = -\sigma^E_3 a_2 R^F_2$.
Thus, outside of $\{ 0 \}$
\begin{equation*}
    R_2^E = (I_{E_2} - \varphi_3 \sigma_3^E) a_2 R_2^F.
\end{equation*}
Then, $R_2^E$ is the standard extension in the sense of \cite{BjAbel}*{Section~6.2},
of \linebreak $(I_{E_2} - \varphi_3 \sigma_3) a_2 R_2^F$. One way to interpret the standard extension
here is that since $R_2^E$ is a pseudomeromorphic $(0,2)$-current defined on all of
$\C^4$, its extension from $\C^4\setminus \{ 0 \}$ is uniquely defined by the
dimension principle.

We have that
\begin{equation*}
    (I_{E_2} - \varphi_3 \sigma_3) a_2 = \frac{1}{|x|^2 +|y|^2+|z|^2+|w|^2}
    \left[\begin{array}{c} w(|y|^2 + |z|^2) \\ z(|x|^2+|w|^2) \\ y(|x|^2+|w|^2) \\ x(|y|^2+|z|^2)
    \end{array}\right].
\end{equation*}
Since $R_2^F = \dbar(1/yw)\wedge\dbar(1/xz)$, see Theorem~\ref{thmbmch}, we get from the
transformation law and Proposition~\ref{proppmantiholo} that $R_2^E$ is the standard
extension of 
\begin{equation*}
    \frac{1}{|x|^2 +|y|^2+|z|^2+|w|^2}
    \left[\begin{array}{c} |z|^2 \dbar\frac{1}{y}\wedge\dbar\frac{1}{xz} \\
    |w|^2 \dbar\frac{1}{yw}\wedge\dbar\frac{1}{x} \\
    |x|^2 \dbar\frac{1}{w}\wedge\dbar\frac{1}{xz} \\
    |y|^2 \dbar\frac{1}{yw}\wedge\dbar\frac{1}{z} \end{array}\right].
\end{equation*}
Using again the transformation law and Proposition~\ref{proppmantiholo}, one gets that
$R_2^E$ is the standard extension of
\begin{equation*}
R_2^E = 
\frac{1}{|z|^2+|w|^2}\left[\begin{array}{c} \overline{z} \\ \overline{w} \\ 0 \\ 0 \end{array}\right]
    \wedge \dbar\frac{1}{y}\wedge\dbar\frac{1}{x} + 
\frac{1}{|x|^2+|y|^2}\left[\begin{array}{c} 0 \\ 0 \\ \overline{x} \\ \overline{y} \end{array}\right]
    \wedge \dbar\frac{1}{w}\wedge\dbar\frac{1}{z}.
\end{equation*}
\end{ex}

\section{The Jacobian determinant of a holomorphic mapping} \label{secthickel}

Throughout this section, we let $f = (f_1,\dots,f_m)$ be a tuple of holomorphic functions, and let $\mathcal{J} := \mathcal{J}(f_1,\dots,f_m)$.
Let $(F,\psi)$ be the Koszul complex of $f$, let $(E,\varphi)$ be a free resolution of $\Ok/\mathcal{J}$, and let $a : (F,\psi) \to (E,\varphi)$
be a morphism of complexes extending the identity morphism $\coker \psi_1 \cong \Ok/\mathcal{J} \cong \coker \varphi_1$, which
exists by Proposition~\ref{propcomplexcomparison}.

\begin{lma} \label{lma:annRf}
    Let $f$, $(E,\varphi)$, and $a$ be as above, let $R^f$ be the Bochner-Martinelli current of $f$, and let
    $R^f_k$ be the part of $R^f$ of bidegree $(0,k)$. For $k < m$, 
    \begin{equation*}
        df_1 \wedge \dots \wedge df_m \wedge a_k R^f_k = 0.
    \end{equation*}
    If $(E,\varphi)$ has length $\leq m$, and if $h$ is a holomorphic function which vanishes on all
    the irreducible components of $Z(f)$ of codimension $m$, then
    \begin{equation*}
        h\, df_1 \wedge \dots \wedge df_m \wedge a_m R^f_m = 0.
    \end{equation*}
\end{lma}

The condition about the length of $(E,\varphi)$ in Theorem~\ref{thm:vasc} comes in due to the following
lemma.

\begin{lma} \label{lma:am}
    Let $f$, $(E,\varphi)$ and $a$ be as above. If $(E,\varphi)$ has length $\leq m$,
    then $a_m$ vanishes on all irreducible components of $Z(f)$ of codimension $< m$.
\end{lma}

\begin{proof}
    Let $V$ be an irreducible component of $Z(f)$ of codimension $< m$.
    Since $\codim Z_m^E \geq m$ by \eqref{eq:codimZkk}, $Z_m^E \cap V$ is nowhere dense in $V$. Thus, by continuity, it is enough
    to prove that $a_m$ vanishes on $V \setminus Z_m^E$. 

    Consider thus a point $z_0 \in V \setminus Z_m^E$, and take a minimal free resolution
    $(K,\eta)$ of $\Ok_{z_0}/\mathcal{J}(f)_{z_0}$, which has length $< m$ since we are outside of $Z_m^E$.
    Let $b : (\bigwedge \Ok_{z_0}^{\oplus m}, \delta_f) \to (K,\eta)$ be a morphism
    induced by the identity morphism as in Proposition~\ref{propcomplexcomparison}.
    Since a minimal free resolution is a direct summand of any free resolution,
    we get an inclusion $i : (K,\eta) \to (E,\varphi)$.
    Thus, one choice of $a' : (\bigwedge \Ok_{z_0}^{\oplus m},\delta_f) \to (E,\varphi)$ would be $a' = i b$.
    Because $(K,\eta)$ has length $< m$, $b_m = 0$, and thus, $a_m' = 0$.
    Hence, there exists one choice of morphism $a : (\bigwedge \Ok_{z_0}^{\oplus m},\delta_f) \to (E,\varphi)$
    such that $a_m$ vanishes near $z_0$. We need to prove that for any choice of $a$, $a_m$ vanishes on $Z(f)$ near $z_0$.
    By Proposition~\ref{propcomplexcomparison} there exists
    $s : (\bigwedge \Ok_{z_0}^{\oplus m},\delta_f) \to (E,\varphi)$ of degree $-1$
    such that
    \begin{equation*}
        a_k - a_k' = \varphi_{k+1} s_k - s_{k-1} (\delta_f)_k.
    \end{equation*}
    In particular, if $k = m$, then $\varphi_{m+1} = 0$ because $(E,\varphi)$ has length $\leq m$,
    so
    \begin{equation*}
        a_m = a_m' + s_{m-1} (\delta_f)_m.
    \end{equation*}
    Thus, $a_m$ vanishes at $Z(f)$ since both
    $a_m'$ and $(\delta_f)_m$ vanish on $Z(f)$.
\end{proof}

\begin{proof}[Proof of Lemma~\ref{lma:annRf}]
    We let $df := df_1 \wedge \dots \wedge df_m$.
    From the proof of \cite{AndIdeals}*{Lemma~8.3} it follows that there
    exists a modification $\pi : \tilde{X} \to (\Cn,0)$, such that
    $\pi^* df \wedge R^{\pi^* f}_k$ is of the form
    \begin{equation*}
        (f_0^{m-1} df_0 \wedge \eta_1 + f_0^m \eta_2) \wedge \dbar \frac{1}{f_0^k},
    \end{equation*}
    where $f_0$ is a single holomorphic function such that $\{ f_0 = 0 \} = \{ \pi^* f = 0 \}$,
    and $\eta_1$ and $\eta_2$ are smooth forms.
    By the Poincar\'e-Lelong formula and the duality theorem,
    this equals $-2\pi i[f_0 = 0] f_0^{m-k} \eta_1$.
    If $k < m$, we thus get that $\pi^* df \wedge R^{\pi^*f}_k = 0$. 
    If $k = m$, then
    \begin{equation*}
        \pi^*(h\,df a_m)\wedge R^{\pi^* f}_m = - (2\pi i) \pi^* (ha_m) \eta_1 \wedge [f_0 = 0],
    \end{equation*}
    which is $0$ since $h a_m$ vanishes on $Z(f)$ by Lemma~\ref{lma:am}, and thus, $\pi^*(ha_m)$ vanishes
    on $\{ f_0 = 0 \} = \{ \pi^* f = 0 \}$. To conclude, $h\, df \wedge a_k R^f_k = 0$ for all $k$
    since
    \begin{equation*}
        h\, df \wedge a_k R^f_k = \pi_*(\pi^* (df \wedge ha_k) R^{\pi^* f}_k) = 0.
    \end{equation*}
\end{proof}

\begin{proof} [Proof of Theorem~\ref{thm:vasc}]
    We first prove that $\mathcal{J}(f) : \mathcal{J}ac(f) \subseteq \mathcal{J}_m(f)$.
    Let $W_m := Z(\mathcal{J}_m(f))$ be the union of the irreducible components of $Z(f)$ of codimension $m$.
    Generically on $W_m$ (more precisely, where it does not intersect any irreducible component of codimension different from $m$)
    $f$ is a complete intersection. Assume that we are at such a generic point $z$ of $W_m$.
    Take $h \in \mathcal{J}(f) : \mathcal{J}ac(f)$. Since $f$ is a complete intersection near $z$,
    it follows from the Poincar\'e-Lelong formula, \cite{CH}*{Section~3.6}, that near $z$,
    \begin{equation} \label{eq:hpl}
        h \frac{1}{(2\pi i)^m}\dbar \frac{1}{f_m}\wedge\dots\wedge\dbar\frac{1}{f_1}\wedge df_1 \wedge \dots \wedge df_m = h [f_1 = \dots = f_m = 0],
    \end{equation}
    where $[f_1 = \dots = f_m = 0]$ is the integration current along $\{ f_1 = \dots = f_m = 0 \}$ with appropriate multiplicities.
    On the other hand, by the duality theorem and the fact that $h \mathcal{J}ac(f) \subseteq \mathcal{J}(f)$,
    \begin{equation} \label{eq:hduality}
         h \dbar \frac{1}{f_m}\wedge\dots\wedge\dbar\frac{1}{f_1}\wedge df_1 \wedge \dots \wedge df_m = 0,
    \end{equation}
    so combining \eqref{eq:hpl} and \eqref{eq:hduality}, $h$ must vanish on $W_m$ near $z$. 
    Thus, $h \in \mathcal{J}_m(f)_z$ for generic $z \in W_m$, i.e., $h$ vanishes generically on $W_m$.
    By continuity, since $\mathcal{J}_m(f) = \mathcal{I}_{W_m}$, we must have $h \in \mathcal{J}_m(f)$.

    We take $(E,\varphi)$, and $a : (\bigwedge \Ok^{\oplus n},\delta_f) \to (E,\varphi)$ as above.
    We now prove the other inclusion, $\mathcal{J}_m(f) \subseteq \mathcal{J}(f) : \mathcal{J}ac(f)$.
    Take $h \in \mathcal{J}_m(f)$. Since $\ann_{\Ok} R^E = \mathcal{J}(f)$, what we want to prove is equivalent
    to that $h\, df \wedge R^E = 0$, where $df := df_1 \wedge \dots \wedge df_m$.
    We get from \eqref{eqRcomparison-component-zero} that
    \begin{equation} \label{eqrebmcomparison}
        R^E_k = a_k R^f_k + \varphi_{k+1} M_{k+1} - \dbar M_k,
    \end{equation}
    where $R^f_k$ is the part of the Bochner-Martinelli current of $f$ of bidegree $(0,k)$,
    and $M_k$ is the part of $M$ with values in $\Hom(\Ok,E_k)$.
    We are done if we can prove that $h\, df$ annihilates
    all the currents of the right-hand side of \eqref{eqrebmcomparison}.

    To begin with, $h\, df$ annihilates $a_k R^f_k$ by Lemma~\ref{lma:annRf}.
    It is thus sufficient to also prove that $h\, df$ annihilates $M_k$ for all $k$.
    Note first that $M_1 = 0$, so we use this as a starting case for a proof by induction.
    By \eqref{eq:Mellk-induction}, outside of $Z_k^E$
    \begin{equation} \label{eq:need-annihilate}
        M_k = \dbar\sigma^E_k M_{k-1} - \sigma^E_k a_{k-1} R^f_{k-1}.
    \end{equation}
    By induction and Lemma~\ref{lma:annRf}, $h\, df$ annihilates both currents on the right-hand side of \eqref{eq:need-annihilate}
    outside of $Z_k^E$, where $\sigma^E_k$ is smooth. Thus,
    $\supp (h\, df \wedge M_k) \subseteq Z^E_k$, and since $h\, df \wedge M_k$ is a $(m,k-1)$-current with support on $Z^E_k$ of codimension 
    $\geq k$ by \eqref{eq:codimZkk}, $h\, df \wedge M_k = 0$ by the dimension principle.
\end{proof}


\begin{bibdiv}
\begin{biblist}

\bib{AndIdeals}{article}{
   author={Andersson, Mats},
   title={Residue currents and ideals of holomorphic functions},
   journal={Bull. Sci. Math.},
   volume={128},
   date={2004},
   number={6},
   pages={481--512},
}

\bib{AndInt2}{article}{
   author={Andersson, Mats},
   title={Integral representation with weights. II. Division and interpolation},
   journal={Math. Z.},
   volume={254},
   date={2006},
   number={2},
   pages={315--332},
}

\bib{AndCH}{article}{
   author={Andersson, Mats},
   title={Uniqueness and factorization of Coleff-Herrera currents},
   journal={Ann. Fac. Sci. Toulouse Math.},
   volume={18},
   date={2009},
   number={4},
   pages={651--661},
}

\bib{AndResCrit}{article}{
   author={Andersson, Mats},
   title={A residue criterion for strong holomorphicity},
   journal={Ark. Mat.},
   volume={48},
   date={2010},
   number={1},
   pages={1--15},
}

\bib{AndNoeth}{article}{
   author={Andersson, Mats},
   title={Coleff-Herrera currents, duality, and Noetherian operators},
   journal={Bull. Soc. Math. France},
   volume={139},
   date={2011},
   number={4},
   pages={535--554},
}

\bib{AS2}{article}{
   author={Andersson, Mats},
   author={Samuelsson, H\aa kan},
   title={A Dolbeault-Grothendieck lemma on complex spaces via Koppelman
   formulas},
   journal={Invent. Math.},
   volume={190},
   date={2012},
   number={2},
   pages={261--297},
}



\bib{AS3}{article}{
   author={Andersson, Mats},
   author={Samuelsson, H{\aa}kan},
   title={Weighted Koppelman formulas and the $\overline\partial$-equation on an analytic space},
   journal={J. Funct. Anal.},
   volume={261},
   date={2011},
   number={3},
   pages={777--802},
}

\bib{ASS}{article}{
   author={Andersson, Mats},
   author={Samuelsson, H{\aa}kan},
   author={Sznajdman, Jacob},
   title={On the Brian\c{c}on-Skoda theorem on a singular variety},
   journal={Ann. Inst. Fourier (Grenoble)},
   volume={60},
   date={2010},
   number={2},
   pages={417--432},
}

\bib{AW1}{article}{
   author={Andersson, Mats},
   author={Wulcan, Elizabeth},
   title={Residue currents with prescribed annihilator ideals},
   journal={Ann. Sci. \'Ecole Norm. Sup.},
   volume={40},
   date={2007},
   number={6},
   pages={985--1007},
}

\bib{AW2}{article}{
   author={Andersson, Mats},
   author={Wulcan, Elizabeth},
   title={Decomposition of residue currents},
   journal={J. Reine Angew. Math.},
   volume={638},
   date={2010},
   pages={103--118},
}

\bib{AW3}{article}{
   author={Andersson, Mats},
   author={Wulcan, Elizabeth},
   title={Global effective versions of the Brian\c{c}on-Skoda-Huneke theorem},
   journal={Invent. Math.},
   volume={200},
   date={2015},
   number={2},
   pages={607--651},
}

\bib{AWPM2}{article}{
   author={Andersson, Mats},
   author={Wulcan, Elizabeth},
   title={Direct images of semi-meromorphic currents},
   journal={Ann. Inst. Fourier (Grenoble)},
   volume={68},
   date={2018},
   number={2},
   pages={875--900},
}

\bib{BjDmod}{article}{
   author={Bj{\"o}rk, Jan-Erik},
   title={$\scr D$-modules and residue currents on complex manifolds},
   status={Preprint},
   date={1996},
}



\bib{BjAbel}{article}{
   author={Bj{\"o}rk, Jan-Erik},
   title={Residues and $\scr D$-modules},
   conference={
      title={The legacy of Niels Henrik Abel},
   },
   book={
      publisher={Springer},
      place={Berlin},
   },
   date={2004},
   pages={605--651},
}



\bib{CH}{book}{
   author={Coleff, Nicolas R.},
   author={Herrera, Miguel E.},
   title={Les courants r\'esiduels associ\'es \`a une forme m\'eromorphe},
   series={Lecture Notes in Mathematics},
   volume={633},
   publisher={Springer},
   place={Berlin},
   date={1978},
}

\bib{DS}{article}{
   author={Dickenstein, A.},
   author={Sessa, C.},
   title={Canonical representatives in moderate cohomology},
   journal={Invent. Math.},
   volume={80},
   date={1985},
   number={3},
   pages={417--434},
}

\bib{DS2}{article}{
   author={Dickenstein, Alicia},
   author={Sessa, Carmen},
   title={R\'esidus de formes m\'eromorphes et cohomologie mod\'er\'ee},
   conference={
      title={G\'eom\'etrie complexe},
      address={Paris},
      date={1992},
   },
   book={
      series={Actualit\'es Sci. Indust.},
      volume={1438},
      publisher={Hermann},
      place={Paris},
   },
   date={1996},
   pages={35--59},
}

\bib{Eis}{book}{
   author={Eisenbud, David},
   title={Commutative algebra},
   series={Graduate Texts in Mathematics},
   volume={150},
   note={With a view toward algebraic geometry},
   publisher={Springer-Verlag},
   place={New York},
   date={1995},
}

\bib{GH}{book}{
   author={Griffiths, Phillip},
   author={Harris, Joseph},
   title={Principles of algebraic geometry},
   note={Pure and Applied Mathematics},
   publisher={Wiley-Interscience [John Wiley \& Sons]},
   place={New York},
   date={1978},
}
		
\bib{Hic}{article}{
   author={Hickel, M.},
   title={Une note \`a propos du Jacobien de $n$ fonctions holomorphes \`a l'origine de $\mathbb C^n$},
   journal={Ann. Polon. Math.},
   volume={94},
   date={2008},
   number={3},
   pages={245--264},
}

\bib{JW}{article}{
   author={Jonsson, Mattias},
   author={Wulcan, Elizabeth},
   title={On Bochner-Martinelli residue currents and their annihilator
   ideals},
   journal={Ann. Inst. Fourier (Grenoble)},
   volume={59},
   date={2009},
   number={6},
   pages={2119--2142}
}

\bib{LJ2}{article}{
   author={Lejeune-Jalabert, M.},
   title={Liaison et r\'esidu},
   conference={
      title={Algebraic geometry},
      address={La R\'abida},
      date={1981},
   },
   book={
      series={Lecture Notes in Math.},
      volume={961},
      publisher={Springer, Berlin},
   },
   date={1982},
   pages={233--240},
}


\bib{Lund1}{article}{
   author={Lundqvist, Johannes},
   title={A local Grothendieck duality theorem for Cohen-Macaulay ideals},
   journal={Math. Scand.},
   volume={111},
   date={2012},
   number={1},
   pages={42--52},
}


\bib{Lund2}{article}{
   author={Lundqvist, Johannes},
   title={A local duality principle for ideals of pure dimension}, 
   status={Preprint},
   date={2013},
   eprint={arXiv:1306.6252 [math.CV]},
   url={http://arxiv.org/abs/1306.6252},
}

\bib{Lar}{article}{
   author={L\"ark\"ang, Richard},
   title={Residue currents associated with weakly holomorphic functions},
   journal={Ark. Mat.},
   volume={50},
   date={2012},
   number={1},
   pages={135--164}
}

\bib{Lar2}{article}{
   author={L{\"a}rk{\"a}ng, Richard},
   title={On the duality theorem on an analytic variety},
   journal={Math. Ann.},
   volume={355},
   date={2013},
   number={1},
   pages={215--234},
}

\bib{Lar4}{article}{
   author={L{\"a}rk{\"a}ng, Richard},
   title={Residue currents with prescribed annihilator ideals on singular
   varieties},
   journal={Math. Z.},
   volume={279},
   date={2015},
   number={1-2},
   pages={333--358},
}

\bib{LarG}{article}{
   author={L\"ark\"ang, Richard},
   title={Explicit versions of the local duality theorem in $\C^n$},
   journal={Illinois J. Math.},
   volume={63},
   date={2019},
   number={1},
   pages={1--45},
}

\bib{LW}{article}{
   author={L{\"a}rk{\"a}ng, Richard},
   author={Wulcan, Elizabeth},
   title={Computing residue currents of monomial ideals using comparison
   formulas},
   journal={Bull. Sci. Math.},
   volume={138},
   date={2014},
   number={3},
   pages={376--392},
}

\bib{LW2}{article}{
   author={L\"{a}rk\"{a}ng, Richard},
   author={Wulcan, Elizabeth},
   title={Residue currents and fundamental cycles},
   journal={Indiana Univ. Math. J.},
   volume={67},
   date={2018},
   number={3},
   pages={1085--1114},
}

\bib{PMScand}{article}{
   author={Passare, Mikael},
   title={Residues, currents, and their relation to ideals of holomorphic
   functions},
   journal={Math. Scand.},
   volume={62},
   date={1988},
   number={1},
   pages={75--152},
}

\bib{PTY}{article}{
   author={Passare, Mikael},
   author={Tsikh, August},
   author={Yger, Alain},
   title={Residue currents of the Bochner-Martinelli type},
   journal={Publ. Mat.},
   volume={44},
   date={2000},
   number={1},
   pages={85--117},
}

\bib{Sz}{article}{
   author={Sznajdman, Jacob},
   title={A Brian\c{c}on-Skoda-type result for a non-reduced analytic space},
   journal={J. Reine Angew. Math.},
   volume={742},
   date={2018},
   pages={1--16},
}



\bib{Tong}{article}{
   author={{To}ng, Yue Lin L.},
   title={Integral representation formulae and Grothendieck residue symbol},
   journal={Amer. J. Math.},
   volume={95},
   date={1973},
   pages={904--917},
}

\bib{TsikhBook}{book}{
   author={{Ts}ikh, A. K.},
   title={Multidimensional residues and their applications},
   series={Translations of Mathematical Monographs},
   volume={103},
   publisher={American Mathematical Society},
   place={Providence, RI},
   date={1992},
}

\bib{Vasc}{article}{
   author={Vasconcelos, Wolmer V.},
   title={The top of a system of equations},
   note={Papers in honor of Jos\'e Adem},
   journal={Bol. Soc. Mat. Mexicana},
   volume={37},
   date={1992},
   number={1-2},
   pages={549--556},
}

\bib{Wi}{article}{
   author={Wiebe, Hartmut},
   title={\"Uber homologische Invarianten lokaler Ringe},
   journal={Math. Ann.},
   volume={179},
   date={1969},
   pages={257--274},
}

\bib{W1}{article}{
   author={Wulcan, Elizabeth},
   title={Products of residue currents of Cauchy-Fantappi\`e-Leray type},
   journal={Ark. Mat.},
   volume={45},
   date={2007},
   number={1},
   pages={157--178},
}
 
\end{biblist}

\end{bibdiv}
\end{document}